\documentclass{amsart}
\usepackage{marktext}
\usepackage{gimac}
\RequirePackage{tikz-cd}

\newcommand{\GI}[2][]{\sidenote[colback=yellow!20]{\textbf{GI\xspace #1:} #2}}

\renewcommand{\paragraph}[1]{\par\medskip\noindent\textbf{#1}\hspace{1ex}}

\newcommand{\sfrac}[2]{#1 / #2}
\WithSuffix\newcommand\sqrt*[2][1/2]{#2^{#1}}

\DeclareBoldMathCommand{\bpi}{\pi}
\DeclareMathOperator{\vol}{vol}
\DeclareMathOperator{\rank}{rank}
\DeclareMathOperator{\tor}{tor}
\newdelim{\ip}{\langle}{\rangle}
\newcommand{\ab}{\mathrm{ab}}

\colorlet{darkgreen}{green!50!black}

\renewcommand{\setminus}{-}
\renewcommand{\downarrow}{\to}

\begin{document}
\title[Heat Kernels and Brownian Winding Numbers]{Long Time Asymptotics of Heat Kernels and Brownian Winding Numbers on Manifolds with Boundary}
\author[Geng]{Xi Geng\textsuperscript{1}}
\address{%
  \textsuperscript{1} Department of Mathematical Sciences,
  Carnegie Mellon University,
  Pittsburgh, PA 15213.}
\email{xig@andrew.cmu.edu}
\author[Iyer]{Gautam Iyer\textsuperscript{2}}
\address{%
  \textsuperscript{2} Department of Mathematical Sciences,
  Carnegie Mellon University,
  Pittsburgh, PA 15213.}
\email{gautam@math.cmu.edu}
\thanks{%
  This material is based upon work partially supported by
  the National Science Foundation under grant
  DMS-1252912 to GI,
  and the Center for Nonlinear Analysis.
}
\begin{abstract}
  Let~$M$ be a compact  Riemannian manifold with smooth boundary.
  We obtain the exact long time asymptotic behaviour of the heat kernel on abelian coverings of $M$ with mixed Dirichlet and Neumann boundary conditions.
  As an application, we study the long time behaviour of the abelianized winding of reflected Brownian motions in~$M$.
  In particular, we prove a Gaussian type central limit theorem showing that when rescaled appropriately, the fluctuations of the abelianized winding are normally distributed with an explicit covariance matrix.
\end{abstract}
\maketitle

\section{Introduction.}

Consider a compact Riemannian manifold~$M$ with boundary.
We address the following questions in this paper:
\begin{enumerate}
  \item
    What is the long time asymptotic behaviour of the heat kernel on abelian covering spaces of~$M$, under mixed Dirichlet and Neumann boundary conditions?

  \item
    What is the long time behaviour of the abelianized winding of trajectories of normally reflected Brownian motion~$M$.
\end{enumerate}
Our main results are Theorem~\ref{t:hker} and Theorem~\ref{t:winding}, stated in Sections~\ref{s:mainthm} and~\ref{s:winding} respectively.
In this section we survey the literature and place this paper in the context of existing results.

\subsection{Long Time Behaviour of Heat Kernels on Abelian Covers.}

The short time behaviour of heat kernels has been extensively studied and is relatively well understood (see for instance~\cite{BerlineGetzlerEA92,Grigoryan99} and the references therein).
The exact long time behaviour, on the other hand, is subtly related to global properties of the manifold, and our understanding of it is far from being complete. 
There are several scenarios in which the long time asymptotics can be determined precisely.
The simplest scenario is when the underlying manifold is compact, in which case the long time asymptotics is governed by the bottom spectrum of the Laplace-Beltrami operator. The problem becomes highly non-trivial for non-compact manifolds. 
Li~\cite{Li86} determined the exact long time asymptotics on manifolds with nonnegative Ricci curvature, under a polynomial volume growth assumption. Lott~\cite{Lott92} and Kotani-Sunada~\cite{KotaniSunada00} determined the long time asymptotics on abelian covers of closed manifolds. In a very recent paper, Ledrappier-Lim~\cite{LedrappierLim15} established the exact long time asymptotics of the heat kernel of the universal cover of a negatively curved closed manifold, generalizing the situation for hyperbolic space with constant curvature.
We also mention that for non-compact Riemannian symmetric spaces, Anker-Ji~\cite{AnkerJi01} established matching upper and lower bounds on the long time behaviour of the heat kernel.

Since the work by Lott~\cite{Lott92} and Kotani-Sunada~\cite{KotaniSunada00} is closely related to ours, we describe it briefly here.
Let $M$ be a closed Riemannian manifold, and $\hat M$ be an abelian cover (i.e.\ a covering space whose deck transformation group is abelian).
The main idea in~\cite{Lott92,KotaniSunada00} is an exact representation of the heat kernel $\hat{H}(t,x,y)$, in terms of a compact family of heat kernels of sections of twisted line bundles over~$M$.
Precisely, the representation takes the form
\begin{equation*}
  \hat{H}(t,x,y)=\int_{\mathcal{G}}H_{\chi}(t,x,y) \, d\chi,
\end{equation*}
where $\mathcal{G}$ is a compact Lie group, and $H_\chi(t,x,y)$ is the heat kernel on sections of a twisted line bundle $E_\chi$ over $M$.
(This is described in detail in Section~\ref{s:liftedRep}, below.)
Since $M$ is compact, $H_\chi$ decays exponentially with rate $\lambda_{\chi, 0}$, the principal eigenvalue of the associated Laplacian~$\lap_\chi$.
Thus the long time behaviour of~$\hat H$ can be determined from the behaviour of $\lambda_{\chi, 0}$ near its global minimum.
For closed manifolds, it is easy to see that the global minimum of $\lambda_{\chi, 0}$ is~$0$, and is attained at a non-degenerate critical point.

In the present paper we study abelian covers of manifolds with boundary, and impose (mixed) Dirichlet and Neumann boundary conditions.
Our main result determines the exact long time asymptotic behaviour of the heat kernel (Theorem~\ref{t:hker}) and is stated in Section~\ref{s:mainthm}.
In this case, the main strategy in~\cites{Lott92,KotaniSunada00} can still be used, however, the minimum of $\lambda_{\chi, 0}$ need not be~$0$.
The main difficulty in the proof in our context is precisely understanding the behaviour of~$\lambda_{\chi, 0}$ near the global minimum.

Under a suitable transformation, the above eigenvalue minimization problem can be reformulated directly as follows.
Let $\omega$ be a harmonic $1$-form on $M$ with Neumann boundary conditions, and consider eigenvalue problem
\begin{equation*}
  -\lap\phi_{\omega}-4\pi i\omega\cdot\nabla\phi_{\omega}+4\pi^{2}|\omega|^{2}\phi_{\omega}=\mu_{\omega}\phi_{\omega}\,,
\end{equation*}
with mixed Dirichlet and Neumann boundary conditions.
It turns out that in order to make the strategy of~\cite{Lott92,KotaniSunada00} work, one needs to show that
\begin{inparaenum}[(i)]
  \item the eigenvalue~$\mu_\omega$ above attains the global minimum if and only if the integral of $\omega$ on closed loops is integer valued, and
  \item in this case the minimum is non-degenerate of second order.
\end{inparaenum}
These are the two key ingredients of the proof, and they are formulated in Lemmas~\ref{l:minLambdaChi} and~\ref{l:muBound} below.
Given these lemmas, the main result of this paper (Theorem~\ref{t:hker}) shows that
\begin{equation*}
  \hat H(t, x, y) \approx \frac{C'_\mathcal I(x, y)}{t^{k/2}} \exp \paren[\Big]{ -\mu_0 t - \frac{d'_\mathcal I(x, y)^2 }{t} }\,,
  \quad\text{as } t \to \infty\,.
\end{equation*}
Here $k$ is the rank of the deck transformation group, and~$C'_\mathcal I$, $d'_\mathcal I$ are explicitly defined functions.

\subsection{The Abelianized Winding of Brownian Motion on Manifolds.}

We now turn our attention to studying the winding of Brownian trajectories on manifolds.
The long time asymptotics of Brownian winding numbers is a classical topic which has been investigated in depth.
The first result in this direction is due to Spitzer~\cite{Spitzer58}, who considered a Brownian motion in the punctured plane.
If $\theta(t)$ denotes the total winding angle up to time~$t$, then Spitzer showed
\begin{equation*}
\frac{2\theta(t)}{\log t}
  \xrightarrow[t\to \infty]{w} \xi \,,
\end{equation*}
where $\xi$ is a standard Cauchy distribution.
The reason that the heavy tailed Cauchy distribution appears in the limit is because when the Brownian motion approaches the origin, it generates a large number of windings in a short period of time.

If one looks at exterior disk instead of the punctured plane, then Rudnick and Hu~\cite{RudnickHu87} (see also Rogers and Williams~\cite{RogersWilliams00}) showed that the limiting distribution is now of hyperbolic type.
In planar domains with multiple holes, understanding the winding of Brownian trajectories is complicated by the fact that it is inherently non-abelian if one wants to keep track of the order of winding around different holes.
Abelianized versions of Brownian winding numbers have been studied in~\cites{PitmanYor86,PitmanYor89,GeigerKersting94,TobyWerner95},
and various generalizations in the context of positive recurrent diffusions, Riemann surfaces, and in higher dimensional domains have been studied in~\cites{
  GeigerKersting94,
  LyonsMcKean84,
  Watanabe00
}.

In this paper, we study the abelianized winding of trajectories of normally reflected Brownian motion on a compact Riemannian manifold with boundary.
The techniques used by many of the references cited above are specific to two dimensions and relies on the conformal invariance of Brownian motion in a crucial way.

Our approach studies Brownian winding on manifolds by lifting trajectories to a covering space, and then using the long time asymptotics of the heat kernel established in Theorem~\ref{t:hker}.
Due to the limitations in Theorem~\ref{t:hker} we measure the winding of Brownian trajectories as a class in $\pi_1(M)_\ab$, the \emph{abelianized} fundamental group of~$M$.
By choosing generators of~$\pi_1(M)$, we measure the abelianized winding of Brownian trajectories as a~$\Z^k$-valued process, denoted by $\rho$.
We show (Theorem~\ref{t:winding}, below) that
\begin{equation*}
  \frac{\rho(t)}{t} \xrightarrow[t \to \infty]{p} 0
  \qquad\text{and}\qquad
  \frac{\rho(t)}{\sqrt{t}} \xrightarrow[t \to \infty]{w} \mathcal N(0, \Sigma)\,,
\end{equation*}
for some explicitly computable matrix~$\Sigma$.
Here $\mathcal N(0, \Sigma)$ denotes a normally distributed random variable with mean $0$ and covariance matrix $\Sigma$.
As a result, one can for instance, determine the long time asymptotics of the abelianized winding of Brownian trajectories around a knot in $\R^3$.

We remark, however, that Theorem~\ref{t:winding} can also be proved directly by using a purely probabilistic argument. For completeness, we sketch this proof in~Section~\ref{s:windingDomain}.

\subsection{The Non-abelian Case.}

One limitation of our techniques is that they do not apply to the non-abelian situation.
Studying the winding of Brownian trajectories without abelianization and the long time behaviour of the heat kernel on non-abelian covers (in particular, on non-amenable covers) are much harder questions.
In the discrete context, Lalley~\cite{Lalley93} (see also~\cite{PittetSaloffCoste01}) showed that the $n$-step transition probability $\P(Z_n=g)$ of a finite range random walk $Z_n$ on the Cayley graph of a free group satisfies
\begin{equation*}
\P(Z_{n}=g)\approx C(g)n^{-\sfrac{3}{2}}R^{-n} \,,
  \quad \text{as } n\rightarrow\infty \,.
\end{equation*}
Here $R$ is the radius of convergence of the Greens function. 
In the continuous context, this suggests that the heat kernel on a non-amenable cover of~$M$ decays exponentially faster than the heat kernel on~$M$, which was shown by Chavel-Karp~\cite{ChavelKarp91}. However, to the best of our knowledge, the exact long time asymptotics is only known in the case of the universal cover of a closed negatively curved manifold by the recent work of Ledrappier-Lim~\cite{LedrappierLim15}, and it remains open beyond the hyperbolic regime.

\paragraph{Plan of this paper.}
In Section~\ref{s:mainthm} we state our main result concerning the long time asymptotics of the heat kernel on abelian covers of $M$ (Theorem~\ref{t:hker}).
In Section~\ref{s:winding} we state our main result concerning the long time behaviour of winding of reflected Brownian motion on~$M$ (Theorem~\ref{t:winding}).
We prove these results in Sections~\ref{s:mainproof} and~\ref{s:pfwinding} respectively.

\section{Long Time Behaviour of the Heat Kernel on Abelian Covers.}\label{s:mainthm}

Let $M$ be a compact Riemannian manifold with smooth boundary, and $\hat{M}$ be a Riemannian cover of $M$ with deck transformation group~$G$ and covering map $\bpi$.
We assume throughout this paper that $G$ is a finitely generated abelian group with rank $k \geq 1$, and $M \cong \hat M / G$.
Let $G_T = \tor(G) \subseteq G$ denote the torsion subgroup of $G$, and let $G_F\defeq G/G_T$. 

Let $\lap$ and $\hat{\lap}$ denote the Laplace-Beltrami operator on $M$ and $\hat{M}$ respectively.
Decompose $\partial M$, the boundary of $M$, into two pieces $\partial_N M$ and $\partial_D M$, and let $H(t,p,q)$ be the heat kernel of $\lap$ on $M$ with Dirichlet boundary conditions on $\partial_D M$ and Neumann boundary conditions on $\partial_N M$.
Let $\partial_D \hat M = \bpi^{-1} (\partial_D M)$ and $\bpi^{-1} (\partial_N M)$, and let $\hat{H}(t,x,y)$ be heat kernel of $\hat{\lap}$ on $\hat{M}$ with Dirichlet boundary conditions on~$\partial_D \hat M$, and Neumann boundary conditions on~$\partial_N \hat M$.

Let $\lambda_{0} \geq 0$ be the principal eigenvalue of $-\lap$ with the above boundary conditions.
Since $M$ is compact, the long time asymptotic behaviour of $H$ can be obtained explicitly using standard spectral theory.
The main result of this paper obtains the asymptotic long time behaviour of the heat kernel $\hat H$ on the non-compact covering space $\hat{M}$.

\begin{theorem}\label{t:hker}
  There exist explicit functions $C_{\mathcal I}, d_{\mathcal I}\colon \hat M \times \hat M \to [0, \infty)$ (defined in~\eqref{e:dIDef} and~\eqref{e:CIdef}, below), such that
  \begin{equation}\label{e:hker}
    \lim_{t\to \infty}
      \paren[\Big]{
	t^{k/2}e^{\lambda_{0}t}\hat{H}(t,x,y)
	-\frac{C_{\mathcal{I}}(x,y)}{|G_T|}
	  \exp\paren[\Big]{-\frac{2\pi^{2}d_{\mathcal{I}}^{2}(x,y)}{t}}}
      =0 \,,
  \end{equation}
  uniformly for $x, y \in \hat M$.
  In particular,
for every $x,y\in\hat{M}$, we have 
\begin{equation*}
\lim_{t\rightarrow\infty}t^{k/2}e^{\lambda_{0}t}\hat{H}(t,x,y)=\frac{C_{\mathcal{I}}(x,y)}{|G_T|} \,.
\end{equation*}
\end{theorem}

The definition of the functions $C_{\mathcal I}$ and $d_\mathcal I$ above requires
the construction of an inner product on a certain space of harmonic $1$-forms over $M$.
More precisely, let $\mathcal H^1$, defined by
\begin{equation*}
  \mathcal H^1 \defeq \set{
    \omega \in TM^* \st
      d \omega = 0,\ d^* \omega = 0,
      \text{ and }
      \omega \cdot \nu = 0 \text{ on } \partial M
  }\,,
\end{equation*}
be the space of harmonic $1$-forms on $M$ that are tangential on $\partial M$.
Here $\nu$ denotes the outward pointing unit normal on $\partial M$, and depending on the context $x \cdot y$ denotes the dual pairing between co-tangent and tangent vectors, or the inner-product given by the metric.
By the Hodge theorem we know that $\mathcal H^1$ is isomorphic to
the first de Rham co-homology group on~$M$.

Now define $\mathcal H^1_G \subseteq \mathcal H^1$ by%
\begin{equation}\label{e:H_G Def}
  \mathcal H^1_G = \set[\Big]{\omega \in \mathcal H^1 \st \oint_{\hat \gamma} \bpi^*(\omega) = 0 \text{ for all closed loops } \hat \gamma \subseteq \hat M}\,.
\end{equation}
It is easy to see 
that $\mathcal H^1_G$ is naturally isomorphic%
\footnote{%
  The isomorphism between $\mathcal{H}^1_G$ and $\hom(G; \R)$, the dual of the deck transformation group~$G$, can be described as follows.
  Given $g \in G$, pick a base point $p_0 \in M$, and a pre-image $x_0 \in \bpi^{-1}(p_0)$.
  Now define
  \begin{equation*}
    \varphi_\omega(g) = \int_{x_0}^{g(x_0)} \bpi^*(\omega)\,,
  \end{equation*}
  where the integral is done over any path connecting $x_0$ and $g(x_0)$.
  By definition of $\mathcal H^1_G$, the above integral is independent of the chosen path.
  Moreover, since $\bpi^*(\omega)$ is the pull-back of $\omega$ by the covering projection, it follows that $\varphi_\omega(g)$ is independent of the choice of $p_0$ or $x_0$.
  Thus $\omega \mapsto \varphi_\omega$ gives a canonical homomorphism between $\mathcal H^1_G$ and $\hom(G, \R)$.
  The fact that this is an isomorphism follows from the transitivity of  the action of~$G$ on fibers.
}
to $\hom(G, \R)$, and hence $\dim(\mathcal H^1_G) = k$.
Define an inner-product on $\mathcal H^1_G$ as follows.
Let $\phi_{0}$ be the principal eigenfunction of $-\lap$ with boundary conditions $\phi_0 = 0$ on $\partial_D M$ and $\nu \cdot \grad \varphi_0 = 0$ on $\partial_N M$.
Let $\lambda_0$ be the associated principal eigenvalue, and normalize $\phi_0$ so that $\phi_0 > 0$ in $M$ and $\norm{\phi_0}_{L^2} = 1$.
Define the quadratic form $\mathcal{I} \colon \mathcal H^1_G \to \R$ by 
\begin{equation}\label{e:Idef}
  \mathcal{I}(\omega)
    = 8\pi^{2}\int_{M} \abs{\omega}^{2}\phi_{0}^{2}
      + 8\pi\int_{M}\phi_{0}
      \omega\cdot\nabla g_\omega \,,
\end{equation}
where $g_\omega$ is a%
\footnote{
  Note, since $\lambda_0$ manifestly belongs to the spectrum of $-\lap$, the function $g_\omega$ is not unique.
  Moreover, one has to verify a solvability condition to ensure that solutions to equation~\eqref{e:gomega} exist.
  We do this in Lemma~\ref{l:muBound}, which is proved in Section~\ref{s:mubound}, below.%
}
solution to the equation 
\begin{equation}\label{e:gomega}
  -\lap g_\omega - 4\pi \omega \cdot \grad \phi_{0} = \lambda_{0}g_\omega\,,
\end{equation}
with boundary conditions
\begin{equation}\label{e:gomegaBC}
  g_\omega = 0 \quad\text{on } \partial_D M\,,
  \qquad\text{and}\qquad
  \nu \cdot \grad g_\omega = 0 \quad\text{on } \partial_N M\,.
\end{equation}
In the course of the proof of Theorem~\ref{t:hker}, we will see that~$\mathcal I$ arises naturally as the quadratic form induced by the Hessian of the principal eigenvalue of a family of elliptic operators (see Lemma~\ref{l:muBound}, below).

Using~$\mathcal I$, define an inner-product on $\mathcal H^1_G$ by
\begin{equation*}
  \ip{ \omega,\tau}_{\mathcal{I}}\defeq\frac{1}{4}\paren[\big]{\mathcal{I}(\omega+\tau)-\mathcal{I}(\omega-\tau)},\ \ \ \omega,\tau\in\mathcal H^1_G\,.
\end{equation*}
We will show (Lemma~\ref{l:muBound}, below) that the function $\mathcal{I}(\omega)$ is well-defined, and that $\ip{\cdot, \cdot}_{\mathcal I}$ is a positive definite inner-product on $\mathcal H^1_G$.
We remark, however, that under Neumann boundary conditions (i.e.\ if~$\partial_D M = \emptyset$), $\lambda_0 = 0$, $\phi_0$ is constant and $\lambda_0 = 0$.
Hence, under Neumann boundary conditions $\ip{\cdot, \cdot}_\mathcal I$ is simply the (normalized) $L^2$ inner-product (see also Remark~\ref{r:neumann}, below).

Now, to define the distance function $d_{\mathcal{I}}:\hat{M}\times\hat{M}\rightarrow\R$ appearing in Theorem~\ref{t:hker}, we take  $x,y\in\hat{M}$ and define $\xi_{x,y}\in (\mathcal H^1_G)^* \defeq \hom (\mathcal H^1_G; \R)$
by 
\begin{equation}\label{e:xiDef}
\xi_{x,y}(\omega)\defeq\int_{x}^y \bpi^*(\omega) \,,
\end{equation}
where the integral is taken over any any smooth path in~$\hat M$ joining $x$ and $y$.
By definition of $\mathcal H^1_G$, the above integral is independent of the choice of path joining $x$ and $y$.
We will show that the function $d_{\mathcal{I}}:\hat{M}\times\hat{M}\rightarrow\R$ is given by
\begin{equation}\label{e:dIDef}
  d_{\mathcal{I}}(x,y)
    \defeq
    \norm{\xi_{x,y}}_{\mathcal{I}^*} 
    = \sup_{\substack{\omega \in \mathcal H^1_G,\\ \norm{\omega}_\mathcal I = 1}}
    \xi_{x,y}(\omega)\,,
    \quad \text{for } x,y\in\hat{M} \,.
\end{equation}
Here $\norm{\cdot}_{\mathcal I^*}$ denotes the norm on the dual space $(\mathcal H^1_G)^*$ obtained by dualising the inner product~$\ip{\cdot, \cdot}_{\mathcal I}$.

Finally, to define $C_\mathcal I$, we let
\begin{equation}\label{e:h1z}
  \mathcal H^1_\Z \defeq \set[\Big]{ \omega \in \mathcal H^1_G \st \oint_\gamma \omega \in\mathbb{Z},\ \text{for all closed loops } \gamma \subseteq M }\,.
\end{equation}
Clearly $\mathcal H^1_\Z$ is isomorphic to $\Z^k$, and hence we can find $\omega_1, \dots, \omega_k\in\mathcal{H}^1_\mathbb{Z}$ which form a basis of $\mathcal H^1_\Z$.
We will show that $C_\mathcal I$ is given by
\begin{equation}\label{e:CIdef}
  C_{\mathcal{I}}(x,y)
    = (2\pi)^{\sfrac{k}{2}} \abs[\Big]{\det\paren[\Big]{\paren[\big]{\ip{\omega_{i},\omega_{j}}_{\mathcal{I}}}_{1\leq i,j\leq k}}}^{-1/2} \phi_{0}(\bpi(x))\phi_{0}(\bpi(y))\,.
\end{equation}

Notice that the value of $C_\mathcal I(x, y)$ doe not depend on the choice of the basis $(\omega_1, \dots, \omega_k)$.
Indeed, if $(\omega_1', \dots, \omega_k')$ is another such basis of the $\Z$-module $\mathcal H^1_\Z$, since the change-of-basis matrix belongs to $GL(k, \Z)$, it must have determinant~$\pm 1$.

We conclude this section by making a few remarks on simple and illustrative special cases.

\begin{remark}[Neumann boundary conditions]\label{r:neumann}
If Neumann boundary conditions are imposed on all of $\partial M$ (i.e.\ $\partial_D M = \emptyset$), then the definitions of $C_\mathcal I$ and $d_\mathcal I$ simplify considerably.
First, as mentioned earlier, under Neumann boundary conditions we have
\begin{equation*}
  \lambda_0 = 0 \,,
  \qquad\text{and}\qquad
  \phi_{0} \equiv \vol (M)^{-1/2} \,,
\end{equation*}
and hence
\begin{equation}\label{e:INeumann}
  \ip{\omega, \tau}_{\mathcal I}
    = \frac{8\pi^{2}}{\vol (M)}\int_{M} \omega \cdot \tau \,,
\end{equation}
is a multiple of the standard $L^2$ inner-product.
Above $\omega \cdot \tau$ denotes the inner-product on $1$-forms inherited from the metric on~$M$.
In this case
\begin{equation*}
  d_\mathcal I(x, y)
    = \paren[\Big]{ \frac{\vol(M)}{8 \pi^2} }^{1/2}
      \sup_{\substack{\omega \in \mathcal H^1_G\\ \norm{\omega}_{L^2(M)} = 1}}
      \int_x^y \bpi^*(\omega) \,,
\end{equation*}
and
\begin{equation*}
  C_\mathcal I(x, y)
    = \frac{(2\pi)^{\sfrac{k}{2}}}{\vol(M)}
      \abs[\Big]{\det\paren[\Big]{\paren[\big]{\langle\omega_{i},\omega_{j}\rangle_{\mathcal{I}}}_{1\leq i,j\leq k}}}^{-1/2} \,
\end{equation*}
is a constant independent of $x,y\in\hat{M}$.

Note that under Neumann boundary conditions the heat kernel $\hat{H}(t,x,y)$ on the covering space $\hat M$ decays like $t^{-k/2}$ as $t \to \infty$.
In contrast, if Dirichlet boundary conditions are imposed on part of the boundary (i.e.\ $\partial_D M \neq \emptyset$), then we know $\lambda_{0}>0$ and $\phi_0$ is not constant.
In this case, $\ip{\cdot, \cdot}_{\mathcal I}$ is not a constant multiple of the standard $L^2$ inner product, and $\hat{H}(t,x,y)$ decays with rate $t^{-k/2}e^{-\lambda_{0}t}$.
\end{remark}

\begin{remark}[Comparison with the Heat Kernel Decay on~$M$]
  Let~$H$ is the heat kernel of~$\lap$ on $M$.
  Since~$M$ is compact by assumption, the spectral decomposition of~$-\lap$ shows that
  \begin{equation*}
    H(t,p,q)\approx e^{-\lambda_{0}t}\phi_{0}(p)\phi_{0}(q)\,,
    \quad\text{for }
    p,q\in M \,,
    \text{ as } t \to \infty\,.
  \end{equation*}
  Thus, using Theorem~\ref{t:hker} we see
  \begin{equation*}
    \lim_{t\rightarrow\infty}
      \frac{t^{\sfrac{k}{2}}\hat{H}(t,x,y)}{H(t,\pi(x),\pi(y))}
      = \frac{(2\pi)^{\sfrac{k}{2}}}{|G_T|} \;
	\abs[\Big]{\det\paren[\Big]{\paren[\big]{\langle\omega_{i},\omega_{j}\rangle_{\mathcal{I}}}_{1\leq i,j\leq k}}}^{-1/2}
      \,.
  \end{equation*}
  Namely, the heat kernel~$\hat{H}(t,x,y)$ decays faster than $H(t,p,q)$ by exactly the polynomial factor $t^{-k/2}$.
\end{remark}

\begin{remark}[Computation of $\omega_i$ in planar domains]\label{r:planar}
  Suppose for now that~$M$ is a bounded planar domain with $k$ holes excised, and $\rank(G_F) = k$.
  In this case, the basis $\set{\omega_{1},\cdots,\omega_{k}}$ can be constructed directly by solving some boundary value problems.
  Indeed, choose $(p_j, q_j)$ inside the $j^\text{th}$ excised hole and define the harmonic form $\tau_j$ by
  \begin{equation}\label{e:taui}
    \tau_{j}\defeq
      \frac{1}{2\pi}
      \paren[\Big]{\frac{
	  (p - p_{j}) \, dq - (q - q_{j} ) \, dp
	}{
	  (p - p_j)^2 + (q - q_j)^2
      }}\,.
  \end{equation}
  Define $\phi_j \colon M \to \R$  to be the solution of the PDE
  \begin{equation*}
    \left\{
      \begin{alignedat}{2}
	\span
	  -\lap \phi_j = 0 
	  &\qquad& \text{in } M\,,
	\\
	\span
	  \partial_\nu \phi_j = \tau_j \cdot \nu
	  && \text{on } \partial M\,.
      \end{alignedat}
    \right.
  \end{equation*}
  Then $\omega_j$ is given by
  \begin{equation*}
    \omega_j = \tau_j + d\phi_j\,.
  \end{equation*}
\end{remark}

\section{The Abelianized Winding of Brownian Motion on Manifolds.}\label{s:winding}

We now study the asymptotic behaviour of the (abelianized) winding of trajectories of reflected Brownian motion on the manifold $M$.
The winding of these trajectories can be naturally quantified by lifting them to the universal cover.
More precisely, let $\bar M$ be the universal cover of $M$, and recall that the fundamental group $\pi_1(M)$ acts on $\bar M$ as deck transformations.
Fix a fundamental domain $\bar U \subseteq \bar M$, and for each $g\in\pi_1(M)$ define $\bar U_g$ to be the image of $\bar U$ under the action of $g$.
Also, define $\bar{\bm{g}} \colon \bar M \to \pi_1(M)$ by
\begin{equation*}
  \bar{\bm{g}}(x) = g	\quad \text{if } x \in U_g\,.
\end{equation*}

Now given a reflected Brownian motion~$W$ in $M$ with normal reflection at the boundary, let $\bar W$ be the unique lift of $W$ to~$\bar M$ starting in $\bar U$.
Define $\bar \rho(t) = \bar{\bm{g}}( \bar W_t ) \in \pi_1(M)$.
That is, $\bar \rho(t)$ is unique element of $\pi_1(M)$ such that $\bar W(t) \in \bar U_{\bar \rho(t)}$.
Note that $\bar \rho(t)$ measures the winding of the trajectory of~$W$ up to time~$t$.

Our main result of Theorem~\ref{t:hker} will enable us to study the asymptotic behaviour of the projection of $\bar \rho$ to the abelianized fundamental group~$\pi_1(M)_\ab$.
We know that
\begin{equation*}
  G \defeq {}^{\textstyle \pi_1(M)_\ab}  \Big/ \:_{\textstyle \tor(\pi_1(M)_\ab)}
\end{equation*}
is a free abelian group of finite rank, and we let~$k = \rank(G)$.
Let $\pi_G\colon \pi_1(M) \to G$ be the projection of the fundamental group of $M$ onto $G$.
Fix a choice of loops $\gamma_1$, \dots, $\gamma_k \in \pi_1(M)$ so that $\pi_G(\gamma_1)$, \dots, $\pi_G(\gamma_k)$ form a basis of $G$.

\begin{definition}\label{d:winding}
  The $\mathbb{Z}^k$-\textit{valued winding number} of $W$ is defined to be the coordinate process of $\pi_G(\bar \rho(t))$ with respect to the basis $\pi_G(\gamma_1)$, \dots, $\pi_G(\gamma_k)$.
  Explicitly, we say $\rho(t) = (\rho_1(t), \dots, \rho_k(t)) \in \Z^k$ if
  \begin{equation*}
    \pi_G(\bar \rho(t)) = \sum_{i = 1}^k \rho_i(t) \pi_G(\gamma_i)\,.
  \end{equation*}
\end{definition}

Note that the $\Z^k$-valued winding number defined above depends on the choice of basis~$\gamma_1$, \dots, $\gamma_k$.
If $M$ is a planar domain with $k$ holes, we can choose $\gamma_i$ to be a loop that only winds around the $i^\text{th}$ hole once.
In this case, $\rho_i(t)$ is the number of times the trajectory of $W$ winds around the $i^\text{th}$ hole in time $t$.

Our main result concerning the asymptotic long time behaviour of~$\rho$ can be stated as follows.

\begin{theorem}\label{t:winding}
  Let $W$ be a normally reflected Brownian motion in $M$, and $\rho$ be its $\Z^k$ valued winding number (as in Definition~\ref{d:winding}).
  Then, there exists a positive definite, explicitly computable covariance matrix~$\Sigma$ (defined in~\eqref{e:sigmadef}, below) such that
  \begin{equation}\label{e:rhoLim}
    \frac{\rho(t)}{t} \xrightarrow{p} 0
    \qquad\text{and}\qquad
    \frac{\rho(t)}{\sqrt{t}} \xrightarrow{w} \mathcal N(0, \Sigma)\,.
  \end{equation}
  Here $\mathcal N(0, \Sigma)$ denotes a normally distributed random variable with mean $0$ and covariance matrix $\Sigma$.
\end{theorem}

\medskip

We now define the covariance matrix~$\Sigma$ appearing in Theorem~\ref{t:winding}.
Given $\omega, \in \mathcal H^1$ define the map $\varphi_\omega \in \hom(\pi_1(M), \R)$ by
\begin{equation*}
  \varphi_\omega(\gamma) = \int_\gamma \omega \,.
\end{equation*}
It is well known that the map $\omega \mapsto \varphi_\omega$ provides an isomorphism between $\mathcal H^1$ and $\hom( \pi_1(M), \R )$.
Hence there exists a unique dual basis $\omega_1$, \dots, $\omega_k \in \mathcal H^1$ such that
\begin{equation}\label{e:omegaiDef}
  \int_{\gamma_i} \omega_j = \delta_{i,j}\,.
\end{equation}
Now, the covariance matrix~$\Sigma$ appearing in Theorem~\ref{t:winding} is given
\begin{equation}\label{e:sigmadef}
  \Sigma_{i,j} \defeq
    \frac{1}{\vol M}\int_{M} \omega_{i} \cdot \omega_{j} \,.
\end{equation}

The proof of Theorem~\ref{t:winding} follows quite easily from our heat kernel result Theorem~\ref{t:hker},  which will be given in Section~\ref{s:pfwinding} below. We remark, however, that Theorem~\ref{t:winding} can also be proved directly by using a probabilistic argument. For completeness, we sketch this proof in Section~\ref{s:windingDomain}.

  A fundamental example of Theorem~\ref{t:winding} is the case when $M$ is a planar domain with multiple holes.
  In this case, in the limiting Gaussian distribution described in the proposition, the forms $\omega_i$ can be obtained quite explicitly following remark ~\ref{r:planar}.
  The winding of Brownian motion in planar domains is a classical topic which has been studied by many authors~\cites{Spitzer58,PitmanYor86,PitmanYor89,RudnickHu87,RogersWilliams00,LyonsMcKean84,GeigerKersting94,TobyWerner95,Watanabe00}.
  The result by Toby and Werner~\cite{TobyWerner95}, in particular, obtains a law of large numbers type result for the time average of the winding number of an obliquely reflected Brownian motion in a bounded planar domain.
  Under normal reflection our result (Theorem~\ref{t:winding}) is a refinement of Toby and Werner's result.
  Namely, we show that the long time average of the winding number is~$0$, and we prove a Gaussian type central limit theorem for fluctuations around the mean.
  A more detailed comparison with the results of~\cite{TobyWerner95} is in Section~\ref{s:tobyWerner}, below.

\begin{remark}[An explicit calculation in the annulus]\label{r:annulus}
  When $M \subseteq \R^2$ is an annulus the covariance matrix $\Sigma$ can be computed explicitly.
  Explicitly, for $0 < r_1 < r_2$ and let
  \begin{equation*}
    A \defeq\set[\big]{ p\in\R^{2} \st r_{1}<|p|<r_{2}} 
  \end{equation*}
  be the annulus with inner radius $r_{1}$ and outer radius $r_{2}$.
  In this case, $k=1$ and define $\rho(t)$ is simply the integer-valued winding number of the reflected Brownian motion in $A$ with respect to the inner hole.
  Now $k = 1$ and the one form $\omega_1$ can be obtained from Remark~\ref{r:neumann}.
  Explicitly, we choose $p_1 = q_1 = 0$, and define $\tau_1$ by~\eqref{e:taui}.
  Now $\tau_1 \cdot \nu = 0$ on $\partial M$, forcing $\phi_1 = 0$ and hence $\omega_1 = \tau_1$.
  Thus Theorem~\ref{t:winding} shows that $\rho(t) / \sqrt{t} \to \mathcal N(0, \Sigma)$ weakly as $t \to \infty$.
  Moreover equation~\eqref{e:sigmadef} and~\eqref{e:INeumann} show that $\Sigma$ is the $1 \times 1$ matrix $(\sigma^2)$ where
  \begin{equation}\label{e:sigmaAnnulus}
    \sigma^{2}
      =\frac{1}{\vol A}\int_{A}\abs{\omega_1}^{2}
      =\frac{1}{2\pi^2(r_{2}^{2}-r_{1}^{2})}\log\paren[\Big]{\frac{r_{2}}{r_{1}}} \,.
  \end{equation}

  We remark, however, that in this case a finer asymptotic result is available.
  Namely, Wen~\cite{Wen17} shows that for large time
  \begin{equation*}
    \var( \rho(t) ) \approx
    \frac{1}{4\pi^2}\paren[\Big]{\ln^2 \paren[\big]{\frac{r_2}{r_1}} - \ln^2 \paren[\big]{\frac{r_1}{r_0}}}
    +
    \frac{\ln \paren{ \sfrac{r_2}{r_1} }}{2\pi^2(r_2^2 - r_1^2)} 
    \paren[\Big]{ t - \frac{r_2^2 - r_0^2}{2} + r_1^2 \ln \paren[\big]{ \frac{r_2}{r_0} }}
  \end{equation*}
  where $r_0 = \abs{W_0}$ is the radial coordinate of the starting point.
  Note Theorem~\ref{t:hker} only shows $\var \rho(t) / t \to \sigma^2$ as $t \to \infty$.
  Wen's result above goes further by providing explicit limit for $\var\rho(t) - \sigma^2 t$ as $t \to \infty$.
\end{remark}

\begin{remark}[Winding in Knot Compliments]
  Another interesting example is the winding of 3D Brownian motion around knots.
  Recall that a knot $K$ is an embedding of $S^{1}$ into $\R^{3}$.
  A basic topological invariant of a knot $K$ is the fundamental group
  $\pi_{1}(\R^{3}\setminus K)$ which is known as the \textit{knot
    group} of $K$. The study of the fundamental group $\pi_{1}(\R^{3}\setminus K)$
  is important for the classification of knots and has significant applications
  in mathematical physics.
  It is well known that the abelianized fundamental group of $\R^{3}\setminus K)$ is always $\Z$.

  Let $K$ be a knot in $\R^{3}$.
  Consider the domain $M=\Omega\setminus N_{K}$,
  where $N$ is a small tubular neighborhood of $K$ and $\Omega$ is
  a large bounded domain (a ball for instance) containing $N_{K}$.
  Let $W(t)$ be a reflected Brownian motion in $M$, and define $\rho(t)$
  to be the $\Z$-valued winding number of $W$ with respect
  to a fixed generator of $\pi_{1}(M)_{\ab}$.
  Now $\rho(t)$ contains information about the entanglement of $W(t)$ with the knot $K$.
  Theorem~\ref{t:winding} applies in this context, and shows that the long time behaviour of $\rho$ is Gaussian with mean~$0$ and covariance given by~\eqref{e:sigmadef}.

    In some cases, the generator of $\pi_1(M)_{\ab}$ (which was used above in defining~$\rho$) can be written down explicitly.
    For instance, consider the $(m,n)$-\textit{torus knot}, $K=K_{m,n}$, defined by $S^{1}\ni z\mapsto(z^{m},z^{n})\in S^{1}\times S^{1}$ where $\gcd(m,n)=1$.
    Then $\pi_{1}(M)$ is isomorphic to the free group with two generators $a$ and $b$, modulo the relation $a^{m}=b^{n}$.
    Here $a$ represents a meridional circle inside the open solid torus and $b$ represents a longitudinal circle winding around the torus in the exterior.
    In this case, a generator of $\pi_1(M)_{\ab}$ is $a^{n'} b^{m'}$, where $m', n'$ are integers such that $m m'+n m'=1$.
    (The existence of such an $m'$ and $n'$ is guaranteed since $\gcd(m,n) = 1$ by assumption.)
    Now $a^{n'} b^{m'}$ represents a unit winding around the knot $K$, and $\rho(t)$ describes the total number of windings around $K$. 
\end{remark}

\section{Proof of Theorem \ref{t:hker}.}\label{s:mainproof}
The main tool used in the proof of Theorem~\ref{t:hker} is an integral representation due to Lott~\cite{Lott92} and Kotani-Sunada~\cite{KotaniSunada00}.
Note that heat kernel $H$ on $M$ can be easily computed in terms of the heat kernel~$\hat H$ on the cover $\hat M$ using the identity
\begin{equation}\label{e:HinTermsOfHatH}
  H(t,p,q)=\sum_{y\in\bpi^{-1}(q)}\hat{H}(t,x,y)\,,
\end{equation}
for any $x\in\bpi^{-1}(p)$.
Seminal work of Lott~\cite{Lott92} and Kotani-Sunada~\cite{KotaniSunada00} address an inverse representation where $\hat{H}(t,x,y)$ is expressed as the integral of a compact family of heat kernels on twisted bundles over $M$.
Since $M$ is compact, the long time behaviour of the these twisted heat kernels is governed by the principal eigenvalue of the associated twisted Laplacian.
Thus, using the integral representation in~\cites{Lott92,KotaniSunada00}, the long time behaviour of $\hat H$ can be deduced by studying the behaviour of the above principal eigenvalues near the maximum.

In the case where only Neumann boundary conditions are imposed on $\partial M$ (i.e.\ if $\partial_D M = \emptyset$), the proof in~\cites{Lott92,KotaniSunada00} can be adapted easily.
If, however, there is a portion of the boundary where a Dirichlet boundary condition is imposed (i.e.\ if $\partial_D M \neq \emptyset$), then one requires finer spectral analysis 
than that is available in~\cites{Lott92,KotaniSunada00}
The key new ingredient lies in understanding the behaviour of the principal eigenvalue of twisted Laplacians.

\paragraph{Plan of this section.}
In Section~\ref{s:liftedRep} we describe the Lott / Kotani-Sunada representation of the lifted heat kernels.
In Section~\ref{ss:hkerProof} we use this representation to prove Theorem~\ref{t:hker}, modulo two key lemmas (Lemmas~\ref{l:minLambdaChi} and~\ref{l:muBound}, below) concerning the principal eigenvalue of the twisted Laplacian.
Finally in Sections~\ref{s:lambdamin} and~\ref{s:mubound} we prove Lemmas~\ref{l:minLambdaChi} and~\ref{l:muBound} respectively.

\subsection{A Representation of the Lifted Heat Kernel.}\label{s:liftedRep}

We begin by describing the Lott~\cite{Lott92} / Kotani-Sunada~\cites{KotaniSunada00} representation of the heat kernel $\hat H$.
Let $S^1 = \set{ z \in \C \st \abs{z} = 1 }$ be the unit circle and let
\begin{equation*}
  \mathcal{G}\defeq\hom (G; S^1)\,,
\end{equation*}
be the space of one dimensional unitary representations of $G$.
We know that $\mathcal G$ is isomorphic to $(S^1)^k$, and hence is a compact Lie group with a unique normalized Haar measure.

Given $\chi\in\mathcal{G}$, define an equivalence relation on $\hat M \times \C$ by 
\begin{equation*}
  (x, \zeta) \sim (g(x), \chi(g) \zeta)
  \quad\text{for all } g \in G\,,
\end{equation*}
and let $E_{\chi}$ be the quotient space $\hat{M}\times\C/ {\sim}$.
Since the action of $G$ on fibers 
is transitive, it follows that $E_\chi$ is a complex line bundle on $M$.

Let $C^{\infty}(E_{\chi})$ be the space of smooth sections of $E_{\chi}$.
Note that elements of $C^{\infty}(E_{\chi})$ can be identified with smooth functions $s \colon \hat M \to \C$ which satisfy the twisting condition
\begin{equation}\label{e:twistingCondition}
  s(g(x))=\chi(g) s(x) \,,
  \quad \forall x\in\hat{M},\ g\in G \,.
\end{equation}
Since $\bpi\colon \hat{M}\to M$ is a local isometry and $G$ acts on $\hat{M}$ by isometries, $E_{\chi}$ carries a natural connection induced by the Riemannian metric on $M$.
Let $\lap_{\chi}$ be the associated Laplacian acting on sections of $E_{\chi}$.
If we impose homogeneous Dirichlet boundary conditions on $\partial_D \hat M$ and homogeneous Neumann boundary conditions on $\partial_N \hat M)$, then the operator $-\lap_{\chi}$ is a self-adjoint positive-definite on $L^{2}(E_{\chi})$.
To write this in terms of sections on $\hat M$, define the space $\mathcal D_\chi$ by
\begin{equation}\label{e:dchi}
  \begin{split}
  \mathcal D_\chi =
    \bigl\{ s \in C^\infty(\hat M, \C)
      \st[\big] & s \text{ satisfies~\eqref{e:twistingCondition}} \,,
	\ s = 0 \text{ on } \partial_D \hat M\,,
      \\
	& \text{and } \nu \cdot \grad s = 0 \text{ on } \partial_N \hat M
    \bigr\}\,.
  \end{split}
\end{equation}
Now $\lap_\chi$ is simply the restriction of the usual Laplacian $\hat \lap$ on $\hat M$, and the $L^2$ inner-product is given by
\begin{equation}\label{e:l2twisted}
  \ip{s_{1},s_{2}}_{L^{2}}
    \defeq
    \int_{M} s_{1}(x_p) \, \overline{s_{2}(x_p)} \, dp \,,
\end{equation}
for $s_{1}, s_{2}\in\mathcal{D}_{\chi}$.
Here for each $p\in M$, $x_{p}$ is a any point in the fiber $\bpi^{-1}(p)$ such that the function $p \mapsto x_p$ is measurable.
The twisting condition \eqref{e:twistingCondition} ensures that \eqref{e:l2twisted} is independent of the choice of $x_p$.

\begin{remark}
  When $\chi \equiv \one$ is the trivial representation, $E_{\chi}$ is the trivial line bundle $M\times\C$, and $\lap_{\chi}$ is the standard Laplacian $\lap$ on $M$.
  When $\chi \not\equiv \one$, $E_{\chi}$ is diffeomorphic to the trivial line bundle, as one can construct a non-vanishing section easily (c.f. \eqref{e:cannonicalSection} below).
  However, $E_\chi$ is \emph{not} isometric to the trivial line bundle, and the use of $E_\chi$ is in the structure of the twisted Laplacian $\lap_{\chi}$, which differs from the standard Laplacian on $M$.
\end{remark}

Let $H_{\chi}(t,x,y)$ be the heat kernel of $-\lap_{\chi}$ on $E_\chi$ (see~\cite{BerlineGetzlerEA92} for the general construction of heat kernels on vector bundles).
As before, we can view $H_\chi$ as a function on $(0,\infty)\times\hat{M}\times\hat{M}$ that satisfies the twisting conditions
\begin{equation*}
  H_{\chi}(t, g(x), y) = \chi(g) \, H_{\chi}(t,x,y)\,,
  \quad\text{and}\quad
  H_{\chi}(t, x, g(y) ) = \overline{\chi(g)} \, H_{\chi}(t,x,y)\,.
\end{equation*}
The Lott~\cite{Lott92} and Kotani-Sunada~\cite{KotaniSunada00} representation expresses~$\hat H$ in terms of $H_\chi$, and allowing us to use properties of $H_\chi$ to deduce properties of~$\hat H$.

\begin{lemma}[Lott, Kotani-Sunada]\label{l:keyRepresentation}
  The heat kernel $\hat H$ on $\hat M$ satisfies the identity
  \begin{equation}\label{e:keyRepresentation}
    \hat{H}(t,x,y)=\int_{\mathcal{G}}H_{\chi}(t,x,y) \, d\chi \,,
  \end{equation}
  where the integral is performed with respect to the normalized Haar measure $d\chi$ on~$\mathcal G$.
\end{lemma}
\begin{proof}
  Since a full proof can be found in \cite[Proposition 38]{Lott92},  and \cite[Lemma 3.1]{KotaniSunada00}, we only provide a short formal derivation.
  Suppose $\hat H$ is defined by~\eqref{e:keyRepresentation}.
  Clearly $\hat H$ satisfies the heat equation with Dirichlet boundary conditions on $\partial_D \hat M$ and Neumann boundary conditions on $\partial_N \hat M)$.
  For initial  data observe
  \begin{equation*}
    H_\chi(0, x, y) = \sum_{g \in G} \overline{\chi(g)} \,  \delta_{g(x)} (y)\,,
  \end{equation*}
  where $\delta_{g(x)}$ denotes the Dirac delta function at $g(x)$.
  Integrating over $\mathcal G$ and using the orthogonality property
  \begin{equation*}
    \int_{\mathcal G} \chi(g) \, d\chi =
      \begin{cases}
	1 & g = \mathrm{Id}\\
	0 & g \neq \mathrm{Id}\,,
      \end{cases}
  \end{equation*}
  we see that $\hat H(0, x, y) = \delta_x(y)$, and hence $\hat H$ must be the heat kernel on $\hat M$.
\end{proof}

\begin{remark}
  \GI[2018-02-23]{Maybe tweak more later.}
  The integral representation~\eqref{e:keyRepresentation} is similar to Fourier transform and inversion.
  Indeed, for each $\chi \in \mathcal G$, it is easy to see that
  \begin{equation*}
    H_\chi(t, x, y) = \sum_{g \in G} \chi(g) \hat H(t, x, g(y))\,.
  \end{equation*}
  One can view $\mathcal{G}\ni\chi\mapsto H_\chi$ as a Fourier transform of $\hat{H}$, and equation~\eqref{e:keyRepresentation} gives the Fourier inversion formula.
\end{remark}


\subsection{Proof of the Heat Kernel Asymptotics (Theorem~\ref{t:hker}).}\label{ss:hkerProof}

The representation~\eqref{e:keyRepresentation} allows us to study the long time behaviour of $\hat H$ using the long time behaviour of $H_\chi$.
Since $M$ is compact, the long time behaviour of the heat kernels $H_\chi$ can be studied by spectral theory. More precisely, the twisted Laplacian $\Delta_\chi$ admits a sequence of eigenvalues 
$$0\leqslant \lambda_{\chi,1}\leqslant\lambda_{\chi,2}\leqslant\cdots\leqslant\lambda_{\chi,j}\leqslant\cdots\uparrow\infty,$$and a corresponding sequence of 
eigenfunctions $\set{s_{\chi,j} \st j\geq0} \subseteq\mathcal{D}_{\chi}$ which forms an orthonormal basis of $L^{2}(E_{\chi})$.
According to perturbation theory, $\lambda_{\chi,j}$ is smooth in $\chi$, and up to a normalization $s_{\chi,j}$ can be chosen to depend smoothly on $\chi$.
The heat kernel $H_{\chi}(t,x,y)$ can now be written as
\begin{equation}\label{e:expansionHChi}
  H_{\chi}(t,x,y)=\sum_{j=0}^{\infty}e^{-\lambda_{\chi,j}t}s_{\chi,j}(x)\overline{s_{\chi,j}(y)} \,.
\end{equation}Note that since $M$ is compact, the above heat kernel expansion is uniform in $x,y\in\hat{M}$ provided the boundary is smooth. This can be seen from the fact that the eigenfunction $s_{\chi,j}$ is uniformly bounded by a polynomial power of eigenvalue $\lambda_{\chi,j}$, together with Weyl's law on the growth the eigenvalues.  
Combining~\eqref{e:expansionHChi} with Lemma~\ref{l:keyRepresentation}, we obtain
\begin{equation}\label{e:hkerSpectral}
  \hat{H}(t,x,y)=\sum_{j=0}^{\infty}\int_{\mathcal{G}}e^{-\lambda_{\chi,j}t}s_{\chi,j}(x)\overline{s_{\chi,j}(y)}d\chi \,.
\end{equation}

From \eqref{e:hkerSpectral}, it is natural to expect that the long time behaviour of $\hat H$ is controlled by the initial term of the series expansion. In this respect, there are two key ingredients for proving Theorem~\ref{t:hker}. The first key point, which is the content of Lemma~\ref{l:minLambdaChi}, will allow us to see that the integral $\int_{\mathcal{G}}e^{-\lambda_{\chi,0}t}s_{\chi,0}(x)\overline{s_{\chi,0}(y)}d\chi$ concentrates at the trivial representation $\chi=\bf{1}$ when $t$ is large. Having such concentration property, the second key point, which is the content of lemma~\ref{l:muBound}, will then allow us to determine the long time asymptotics of $\hat{H}$ precisely from the rate at which $\lambda_{\chi,0}\rightarrow\lambda_0$ as $\chi\rightarrow\bf{1}\in\mathcal{G}$. Note that when $\chi = \one$ the corresponding eigenvalue $\lambda_{\one, 0}$ is exactly $\lambda_0$, the principal eigenvalue of $-\lap$ on $M$.

\begin{lemma}[Minimizing the principal eigenvalue]\label{l:minLambdaChi}
  The function $\chi \mapsto \lambda_{\chi, 0}$ attains a unique global minimum on $\mathcal G$ at the trivial representation $\chi = \one$.
\end{lemma}

We prove Lemma~\ref{l:minLambdaChi} in Section~\ref{s:lambdamin}, below.
Note that when $\chi=\one$, $\lap_{\chi}$ is simply the standard Laplacian $\lap$ acting on functions on $M$.
If Neumann boundary conditions are imposed on all of $\partial M$ (i.e.\ when $\partial_D M = \emptyset$), $\lambda_{\one, 0} = 0$.
In this case, the proof of Lemma~\ref{l:minLambdaChi} can be adapted from the arguments in~\cite{Sunada89} (see also a direct proof in Section~\ref{s:lambdamin} in the Neumann boundary case). 
If, however, Dirichlet boundary conditions are imposed on a portion of $\partial M$ (i.e.\ $\partial_D M \neq \emptyset$), then $\lambda_{\one, 0} > 0$ and the proof of Lemma~\ref{l:minLambdaChi} requires some work.

In view of ~\eqref{e:hkerSpectral} and Lemma~\ref{l:minLambdaChi}, to determine the long time behaviour of $\hat H$ we also need to understand the rate at which $\lambda_{\chi, 0}$ approaches the global minimum as $\chi \to \one$.
When $G$ is torsion free, we do this by transferring the problem to the linear space $\mathcal H^1_G$.
Explicitly, given $\omega \in \mathcal H^1_G$, we define $\chi_\omega \in \mathcal G$ by
\begin{equation}\label{e:expmap}
  \chi_\omega(g) = \exp\paren[\Big]{ 2 \pi i \int_{x_0}^{g(x_0)} \bpi^*(\omega) }\,,
\end{equation}
for some $x_0 \in \hat M$.
The integral above is done over any smooth path in $\hat M$ joining $x_0$ and $g(x_0)$.
Recall that (Section~\ref{s:mainthm}) for all $\omega \in \mathcal H^1_G$, this integrals is independent of both the path of integration and the choice of $x_0$.
Note that when $G$ is torsion free, the map $\omega \mapsto \chi_\omega$ is a surjective homomorphism between $\mathcal H^1_G$ and $\mathcal G$ whose kernel is precisely $\mathcal H^1_\Z$ defined by \eqref{e:h1z}.
The space $\mathcal{H}^1_G$ can be identified with the Lie algebra of $\mathcal{G}$ and under this identification the map $\omega\mapsto\chi_\omega$ is exactly the exponential map.

Now the rate at which $\lambda_{\chi, 0} \to \lambda_{0}$ as $\chi \to \one \in \mathcal G$ can be obtained from the rate at which $\lambda_{\chi_\omega, 0} \to \lambda_{0}$ as $\omega \to 0 \in \mathcal H^1_G$.
In fact, we claim that the quadratic form induced by the Hessian of the map $\omega \mapsto \lambda_{\chi_\omega, 0}$ at $\omega = 0$ is precisely $\mathcal I(\omega)$ defined by~\eqref{e:Idef},
and this determines the rate at which $\lambda_{\chi_\omega, 0}$ approaches the global minimum~$\lambda_0$.

\begin{lemma}[Positivity of the Hessian]\label{l:muBound}
  For any $\epsilon > 0$, there exists $\delta > 0$ such that if $0 < \abs{\omega} < \delta$ we have
  \begin{equation}\label{e:muBound}
    \abs[\Big]{\lambda_{\chi_\omega, 0} - \lambda_0-\frac{\mathcal{I}(\omega)}{2}}
      < \epsilon\norm{\omega}_{L^2(M)}^2\,,
  \end{equation}
  where $\mathcal I(\omega)$ is defined in~\eqref{e:Idef}.
  Moreover, the map $\omega \mapsto \mathcal I(\omega)$ is a well defined quadratic form, and induces a positive definite inner product on $\mathcal H^1_G$.
\end{lemma}

We point out that the positivity of the quadratic form $\mathcal{I}(\omega)$ is crucial. As mentioned earlier (Remark~\ref{r:neumann}), if only Neumann boundary condition is imposed on  $\partial M$, $\mathcal I(\omega)$ is simply a multiple of the standard $L^2$ inner product on $1$-forms over $M$, whose positivity is straight forward.
The positivity of $\mathcal{I}(\omega)$ in the case of Dirichlet boundary conditions requires some extra work. We prove Lemma~\ref{l:muBound} in Section~\ref{s:mubound}.

Assuming Lemma~\ref{l:minLambdaChi} and Lemma ~\ref{l:muBound} for the moment, we can now prove Theorem~\ref{t:hker}.
We first consider the case when $G$ is torsion free, and will later show how this implies the general case.

\begin{proof}[Proof of Theorem~\ref{t:hker} when $G$ is torsion free]
Note first that Lemma~\ref{l:minLambdaChi} allows us to localize the integral in~\eqref{e:hkerSpectral} to an arbitrarily small neighborhood of the trivial representation $\one$.
More precisely, we claim that for any open neighborhood $R$ of $\one\in\mathcal{G}$, there exist constants $C_1>0$, such that 
  \begin{equation}\label{e:chiLocal}
    \sup_{x,y \in \hat M} \abs[\Big]{ e^{\lambda_0 t} \hat H(x, y, t)
      - \int_R \exp\paren[\Big]{-(\lambda_{\chi, 0} - \lambda_0) t }
	  s_{\chi, 0}(x) \overline{s_{\chi, 0}(y)} \, d\chi
    }
    \leq e^{-C_1 t}.
  \end{equation}
  This in particular implies that the long time behavior of $\hat{H}(t,x,y)$ is determined by the long time behavior of the integral representation around an arbitrarily small neighborhood of $\mathbf{1}\in\mathcal{G}$.

To establish~\eqref{e:chiLocal}, recall that Rayleigh's principle and the strong maximum principle guarantee that $\lambda_{\one,0}$ is simple.
Standard perturbation theory (c.f. \cite{ReedSimon78}, Theorem XII.13) guarantees that when $\chi$ is sufficiently close to~$\one$, the eigenvalue $\lambda_{\chi,0}$ is also simple (i.e. $\lambda_{\chi,0}<\lambda_{\chi,1}$).
Now, by Lemma~\ref{l:minLambdaChi}, we observe
\begin{equation*}
  \lambda'\defeq\min\set[\big]{
    \inf\set{ \lambda_{\chi,1} \st \chi\in\mathcal{G}}
    \,, \ 
    \inf\set{\lambda_{\chi,0} \st \chi\in\mathcal{G}\setminus R}
  }
    > \lambda_{0}.
\end{equation*}
Hence by choosing $C_1 \in ( 0, \lambda'-\lambda_{0})$, we have 
  \begin{multline*}
    \sup_{x,y\in\hat{M}}\Bigl( \abs[\Big]{\sum_{j=1}^{\infty}\int_{\mathcal{G}}e^{-(\lambda_{\chi,j}-\lambda_{0})t}s_{\chi,j}(x)\overline{s_{\chi,j}(y)} \, d\chi}
  \\
    +\abs[\Big]{\int_{\mathcal{G}\setminus R}e^{-(\lambda_{\chi,0}-\lambda_{0})t}s_{\chi,0}(x)\overline{s_{\chi,0}(y)}d\chi} \Bigr)  \leq e^{-C_1 t}
  \end{multline*}
 for all $t$ sufficiently large.
  This immediately implies~\eqref{e:chiLocal}.

For any small neighborhood $R$ of $\one$ as before,  our next task is to convert the integral over~$R$ in~\eqref{e:chiLocal} to an integral over a neighborhood of $0$ in $\mathcal H^1_G$ (the Lie algebra of $\mathcal{G}$) using the exponential map~\eqref{e:expmap}.
  To do this, recall $(\omega_1, \dots, \omega_k)$ was chosen to be a basis of $\mathcal H^1_\Z \subseteq \mathcal H^1_G$.
  Identifying $\mathcal H^1_G$ with $\R^k$ using this basis, we let $d\omega$ denote the pullback of the Lebesgue measure on $\R^k$ to $\mathcal H^1_G$.
  (Equivalently, $d\omega$ is the Haar measure on $\mathcal H^1_G$ normalized so that the parallelogram with sides $\omega_1$, \dots, $\omega_k$ has measure~$1$.)
  Clearly
  \begin{multline}\label{e:intR}
    \int_R \exp\paren[\Big]{-(\lambda_{\chi, 0} - \lambda_0) t }
      s_{\chi, 0}(x) \overline{s_{\chi, 0}(y)} \, d\chi
    \\
    = \int_T \exp\paren[\Big]{-(\mu_\omega - \lambda_0) t }
	s_{\chi_\omega, 0}(x) \overline{s_{\chi_\omega, 0}(y)} \, d\omega\,.
  \end{multline}
  Here $\mu_\omega \defeq \lambda_{\chi_\omega, 0}$ and $T$ is the inverse image of $R$ under the map $\omega \mapsto \chi_\omega$.

  Recall the eigenfunctions $s_{\chi_\omega, 0}$ appearing above are sections of the twisted bundle $E_{\chi_\omega}$. They can be converted to functions on~$M$ using some canonical section $\sigma_{\omega}$. 
  Explicitly, let $x_{0}\in\hat{M}$ be a fixed point, and given $\omega\in \mathcal H^1_G$, define  $\sigma_\omega\colon \hat M \to \C$ by
  \begin{equation}\label{e:cannonicalSection}
    \sigma_{\omega}(x)
      \defeq \exp\paren[\Big]{2\pi i\int_{x_{0}}^{x} \bpi^*(\omega) } \,.
  \end{equation}
  Here $\bpi^*(\omega)$ is the pullback of $\omega$ to $\hat{M}$ via the covering projection~$\bpi$, and the integral above is performed along any smooth path in $\hat M$ joining $x_0$ and $x$.
  By definition of $\mathcal H^1_G$, this integral does not depend on the path of integration.

  Observe that for any $g\in G$ we have
  \begin{equation}\label{e:sigmaOmegaDef}
    \sigma_\omega( g(x))
      = \sigma_\omega(x) \exp\paren[\Big]{ 2 \pi i \int_x^{g(x)} \bpi^*(\omega) }
      = \chi_\omega(g) \sigma_\omega(x)\,,
  \end{equation}
  where $\chi_\omega \in \mathcal G$ is defined in equation~\eqref{e:expmap}.
  Thus $\sigma_\omega$ satisfies the twisting condition~\eqref{e:twistingCondition} and hence can be viewed as a section of $E_{\chi_\omega}$.

  Now define
  \begin{equation*}
    \phi_\omega \defeq \overline{\sigma_\omega} \, s_{\chi_\omega, 0}
  \end{equation*}
  and notice that $\phi_\omega(g(x)) = \phi_\omega(x)$ for all $g \in \mathcal G$.
  This implies $\phi_\omega \circ \bpi = \phi_\omega$, and hence $\phi_\omega$ can be viewed as a (smooth) $\C$-valued function on $M$.
  Consequently, we can now rewrite~\eqref{e:intR} as
  \begin{multline}\label{e:intR2}
    \int_R \exp\paren[\Big]{-(\lambda_{\chi, 0} - \lambda_0) t }
      s_{\chi, 0}(x) \overline{s_{\chi, 0}(y)} \, d\chi
    \\
    = \int_T \exp\paren[\Big]{
	  -(\mu_\omega - \mu_0) t - 2\pi i \xi_{x,y}(\omega)
	}
	\phi_\omega(x) \overline{\phi_\omega(y)}
	\, d\omega\,.
  \end{multline}
  where~$\xi_{x,y}(\omega)$ is defined in~\eqref{e:xiDef}.
  (Of course, when $\omega = 0$, $\chi_\omega = \one$ and hence $\mu_0 = \lambda_0$.)
  Thus, using~\eqref{e:chiLocal}, we have
  \begin{equation}\label{e:omegaLocal}
    \sup_{x,y\in\hat{M}}\abs[\Big]{e^{\lambda_{0}t}\hat{H}(x,y,t) - I_1}
      \leq e^{-C_{1}t}\,,
      \quad\text{for $t$ sufficiently large}\,.
  \end{equation}
  Here 
  \begin{equation*}
	I_1 \defeq \int_{T}\exp\paren[\big]{-\paren{\mu_{\omega}-\mu_{0}}t-2\pi i\xi_{x,y}(\omega)}\phi_{\omega}(x)\overline{\phi_{\omega}(y)} \, d\omega\,,
  \end{equation*}
  and $C_1$ is the constant appearing in~\eqref{e:chiLocal}, and depends on the neighborhood $R$.

  By making the neighborhood~$R$ (and hence also $T$) small, we can ensure that $\phi_\omega$ close to $\phi_0$.
  Moreover, when $\omega$ is close to~$0$, Lemma~\ref{l:muBound} implies $\mu_\omega - \mu_0 \approx \mathcal I(\omega)/2$.
  Thus we claim that for any $\eta > 0$, the neighborhood~$R \ni \one$ can be chosen such that
  \begin{equation}\label{e:triangle}
    \limsup_{t\to \infty} \sup_{x, y \in \hat M}
      t^{k/2} (I_1 - I_2) < \eta \,,
  \end{equation}
  where
  \begin{equation*}
    I_2 \defeq \int_{\mathcal{H}_{G}^{1}}\exp\paren[\Big]{-\frac{1}{2}\mathcal{I}(\omega)t-2\pi i\xi_{x,y}(\omega)}\phi_{0}(x)\overline{\phi_{0}(y)} \, d\omega \,.
  \end{equation*}

  To avoid breaking continuity, we momentarily postpone the proof of~\eqref{e:triangle}.
  Now we see that~\eqref{e:omegaLocal} and~\eqref{e:triangle} combined imply
  \begin{equation}\label{e:HlimI2}
    \lim_{t\to \infty} \paren[\Big]{ t^{k/2} e^{\lambda_0 t} \hat H(t, x, y)
	- t^{k/2} I_2 } = 0
  \end{equation}
  Thus to finish the proof we only need to evaluate~$I_2$ and express it in the form in~\eqref{e:hker}.
  
  To do this,  write $\omega = \sum c_n \omega_n \in \mathcal H^1_G$ and observe
    \begin{equation*}
      \mathcal I(\omega) = \sum_{m, n \leq k} a_{m,n} c_m c_n\,,
      \qquad
      \text{where }
      a_{m,n} = \ip{\omega_m, \omega_n}_\mathcal I\,.
    \end{equation*}
    Let $A$ be the matrix $(a_{m,n})$, and $a_{m,n}^{-1}$ be the $(m,n)$ entry of the matrix $A^{-1}$.
    Consequently
    \begin{align*}
     I_2 &= 
	\phi_0(x) \, \overline{\phi_0(y)} \mathbin{\cdot}
      \\
	&\qquad
	\int_{c \in \R^k}
	      \exp\paren[\Big]{ - \sum_{m,n = 1}^k a_{m,n} c_m c_n t - 2\pi i \sum_{m=1}^k c_m \xi_{x,y}(\omega_m) } \, dc_1 \cdots dc_k
      \\
	&= \phi_0(x) \, \overline{\phi_0(y)}
	   \frac{(2\pi)^{k/2}}{t^{k/2}\det( a_{m,n} )^{1/2}}
	      \exp\paren[\Big]{ - \frac{2\pi^2}{t} \sum_{m,n=1}^k a_{m,n}^{-1} \xi_{x,y}(\omega_m) \xi_{x,y}(\omega_n) } 
      \\
	&= \phi_0(x) \, \overline{\phi_0(y)}
	  \frac{(2\pi)^{k/2}}{t^{k/2}\det( a_{m,n} )^{1/2}}
	      \exp\paren[\Big]{ - \frac{2\pi^2}{t} \norm{\xi_{x,y}}^2_{\mathcal I^*} } \,,
    \end{align*}
    where the second equality followed from the formula for the Fourier transform of the Gaussian.
    Note that when $\omega = 0$, $\sigma_\omega \equiv \one$ and hence $\phi_0 = s_{\one, 0}$ is the principal eigenfunction of $-\lap$ on $M$, viewed as a function on $\hat M$.
    Hence $\phi_0$ is real, and so $\overline{\phi_0} = \phi_0$, and we have
  \begin{equation*}
    I_2 =t^{-\sfrac{k}{2}}C_{\mathcal{I}}(x,y)\exp\paren[\Big]{-\frac{2\pi^{2}d_{\mathcal{I}}^{2}(x,y)}{t}} \,,
  \end{equation*}
  where $C_\mathcal{I}$ is defined by \eqref{e:CIdef}. 
  Combined with~\eqref{e:HlimI2} this finishes the proof of Theorem~\ref{t:hker} when $G$ is torsion free.
  \medskip

  It remains to prove~\eqref{e:triangle}.
  Since
  $\omega \mapsto \phi_\omega$ is continuous,
  there exists a neighborhood $T \ni 0$ such that 
  \begin{equation}\label{e:ctyPhi}
    \sup_{x\in\widehat{M}}\abs[\big]{\phi_{\omega}(x)-\phi_{0}(x)} < \eta
    \quad\text{for all } \omega \in T\,.
  \end{equation}
  Now we know that \eqref{e:omegaLocal} holds with some constant $C_1 = C_1(\eta) >0$ when $t$ is large.
Write 
\begin{equation*}
  t^{\sfrac{k}{2}}(I_1-I_2)=J_{1}+J_{2}+J_{3} \,,
\end{equation*}
where 
\begin{gather*}
J_{1}\defeq t^{\sfrac{k}{2}}\int_{T}\paren[\Big]{e^{-\paren{\mu_{\omega}-\mu_{0}}t}-e^{-\mathcal{I}(\omega)t / 2}}\exp\paren[\big]{-2\pi i\xi_{x,y}(\omega)}\phi_{\omega}(x)\overline{\phi_{\omega}(y)} \, d\omega \,,
\\
J_{2} \defeq t^{\frac{k}{2}}\int_{T}\exp\paren[\Big]{-\frac{1}{2}\mathcal{I}(\omega)t-2\pi i\xi_{x,y}(\omega)}\paren[\Big]{\phi_{\omega}(x)\overline{\phi_{\omega}(y)}-\phi_{0}(x)\overline{\phi_{0}(y)}} \, d\omega \,,
\end{gather*}
and
\begin{equation*}
J_{3}\defeq t^{\sfrac{k}{2}}\int_{\mathcal{H}_{G}^{1}\setminus T}\exp\paren[\Big]{-\frac{1}{2}\mathcal{I}(\omega)t-2\pi i\xi_{x,y}(\omega)}\phi_{0}(x)\overline{\phi_{0}(y)} \, d\omega \,.
\end{equation*}

 First, by Lemma~\ref{l:muBound}, $\mathcal{I}(\omega)$ is a positive definite quadratic form, and hence the Gaussian tail estimate shows there exists $C_2 = C_2(\eta) >0$, such that
\begin{equation*}
|J_3|\leq e^{-C_{2}t}
\end{equation*}
uniformly in $x,y\in\hat{M}$, when $t$ is sufficiently large.  

 Next, by \eqref{e:ctyPhi} and the positivity of the quadratic form $\mathcal{I}(\omega)$, we have 
\begin{equation*}
  \abs{J_{2}}
  \leq C_{3}\eta t^{\sfrac{k}{2}}\int_{T}e^{-\mathcal{I}(\omega)t / 2} \, d\omega
 =C_{3}\eta\int_{\sqrt{t}\cdot T}e^{-\mathcal{I}(v) / 2} \, dv
 \leq C_{4}\eta \,,
\end{equation*}
uniformly in $x,y\in\hat{M}$.

Finally, to estimate $J_1$, first choose $K\subseteq\mathcal{H}^1_G$ compact such that
\begin{equation*}
  \int_{\mathcal{H}_{G}^{1}\setminus K}\exp\paren[\Big]{-\frac{1}{4}\mathcal{I}(v)} \, dv<\eta \,.
\end{equation*}
By using the same change of variables $v=\sqrt{t}\omega$, we  write
\begin{equation*}
J_{1}=J_{1}'+J_{1}'',
\end{equation*}
where 
\begin{align*}
J_{1}' &\defeq\int_{K}\paren[\Big]{\exp\paren[\Big]{-\paren[\Big]{\mu_{v/t^{1/2}}-\mu_{0}}t}-\exp\paren[\Big]{-\frac{1}{2}\mathcal{I}(v)}}
  \\
  &\qquad\qquad\cdot\exp\paren[\Big]{-\frac{2\pi i}{\sqrt{t}}\xi_{x,y}(v)}\phi_{\sfrac{v}{t^{1/2}}}(x)\overline{\phi_{\sfrac{v}{\sqrt*{t}}}(y)} \, dv
\end{align*}
and 
\begin{multline*}
J_{1}'' \defeq\int_{\sqrt{t}\cdot T\setminus K}\paren[\Big]{\exp\paren[\Big]{-\paren[\Big]{\mu_{v/t^{1/2}}-\mu_{0}}t}-\exp\paren[\Big]{-\frac{1}{2}\mathcal{I}(v)}}\\
 \cdot\exp\paren[\Big]{-\frac{2\pi i}{\sqrt{t}}\xi_{x,y}(v)}\phi_{\sfrac{v}{\sqrt*{t}}}(x)\overline{\phi_{\sfrac{v}{\sqrt*{t}}}(y)} \, dv
\end{multline*}
respectively.
By Lemma~\ref{l:muBound}, we know that 
\begin{equation*}
\lim_{t\rightarrow\infty}\paren[\big]{\mu_{v/t^{1/2}}-\mu_{0}}t=\frac{1}{2}\mathcal{I}(v) \,,
\end{equation*}
for every $v\in\mathcal{H}^1_G$.
Therefore, by the dominated convergence theorem, we have 
\begin{equation*}
\lim_{t\rightarrow\infty}\sup_{x,y\in\hat{M}}\abs{J_{1}'}=0 \,.
\end{equation*}
To estimate $J_1''$, choose $\epsilon>0$  such that 
\begin{equation*}
\frac{1}{4}\mathcal{I}(\omega)\geq\epsilon\norm{\omega}_{L^{2}(M)}^{2} \,,
\qquad\text{for all }
\omega\in\mathcal{H}_{G}^{1} \,.
\end{equation*}
For this $\epsilon$, Lemma~\ref{l:muBound} allows us to further assume that $T$ is small enough so that 
\begin{equation*}
\omega\in T\implies\mu_{\omega}-\mu_{0}\geq\frac{1}{2}\mathcal{I}(\omega)-\epsilon\norm{\omega}_{L^{2}(M)}^{2}\geq\frac{1}{4}\mathcal{I}(\omega).
\end{equation*}
In particular, we have 
\begin{equation*}
v\in\sqrt{t}\cdot T\implies\paren[\big]{\mu_{\sfrac{v}{t^{1/2}}}-\mu_{0}}t\geq\frac{1}{4}\mathcal{I}(v).
\end{equation*}
It follows that
\begin{align*}
J_{1}'' & \leq C_{5}\int_{\sqrt{t}\cdot T\setminus K}\paren[\Big]{\exp\paren[\big]{-\paren[\big]{\mu_{v/t^{1/2}}-\mu_{0}}t}+\exp\paren[\Big]{-\frac{1}{2}\mathcal{I}(v)}} \, dv\\
 & \leq2C_{5}\int_{\sqrt{t}\cdot T\setminus K}\exp\paren[\Big]{-\frac{1}{4}\mathcal{I}(v)} \, dv\\
 & \leq2C_{5}\int_{\mathcal{H}_{G}^{1}\setminus K}\exp\paren[\Big]{-\frac{1}{4}\mathcal{I}(v)}dv\\
 & \leq2C_{5}\eta \,,
\end{align*}
uniformly in $x,y\in\hat{M}$. 

Combining the previous estimates, we conclude
\begin{equation*}
  \overline{\lim_{t\rightarrow\infty}}\sup_{x,y\in\hat{M}}\paren[\Big]{t^{k/2}(I_1-I_2)}\leq(C_{4}+2C_{5})\eta \,,
\end{equation*}
and $\eta$ with $\eta / (C_4 + 2 C_5)$ yields~\eqref{e:triangle} as claimed.
\end{proof}

When $G$ is has a torsion subgroup, we prove Theorem~\ref{t:hker} factoring through an intermediate finite cover.

\begin{proof}[Proof of Theorem~\ref{t:hker} when $G$ has a torsion subgroup]
  Since $G$ can be (non-ca\-no\-ni\-cal\-ly) expressed as a direct sum $G_T\oplus G_F$, we define $M_1 = \hat M / G_F$.
  This leads to the covering factorization
  \begin{equation}\label{e:torsionFactorization}
    \begin{tikzcd}
      \hat{M}
	  \arrow[r,"\bpi_F"]
	  \arrow[dr,"\bpi"']
	& M_1 \defeq \hat{M}/G_F
	  \arrow[d,"\bpi_T"]
      \\
      & M\,,
    \end{tikzcd}
  \end{equation}
   where $\bpi_{T}$ and $\bpi_{F}$ have deck transformation groups $G_T$ and  $G_F$ respectively, and $M_{1}$ is compact.

  Recall that $\lambda_{0}$ is the principal eigenvalue of $-\lap$ on $M$, and $\phi_{0}$ is the corresponding $L^2$ normalized eigenfunction.
  Let $\Lambda_{0}$ be the principal eigenvalue of $-\lap_1$ on $M_1$, and $\Phi_{0})$ be the corresponding $L^2$ normalized eigenfunction.
  (Here $\lap_{1}$ is the Laplacian on $M_{1}$.)

  Notice that $\bpi_T^* \phi_0$, the pull back of $\phi_0$ to $M_1$, is an eigenfunction of $-\lap_1$ and $\norm{\bpi_T^* \phi_0}_{L^2(M)} = \abs{G_T}^{1/2}$.
  Thus
  \begin{equation}\label{e:Lambda1EqLambda}
    \Lambda_0 = \lambda_0
    \qquad\text{and}\qquad
    \Phi_{0}=\frac{\bpi_{T}^{*}\phi_{0}}{|G_T|^{1/2}}\,.
  \end{equation}

  Let $\mathcal I_1(\omega_1)$ be the analogue of $\mathcal I$ (defined in equation~\eqref{e:Idef}) for the manifold $M_1$.
  Explicitly,
  \begin{equation*}
     \mathcal{I}_1(\omega_1)
       =8\pi^{2}\int_{M_{1}}|\omega_1|^{2}\Phi_{0}^{2}+8\pi\int_{M_{1}}\Phi_{0} \, \omega_1\cdot\nabla g_{1}\,,
  \end{equation*}
  where $g_1$ is a solution of
  \begin{equation*}
    -\lap g_1 - 4 \pi \omega_1 \cdot \grad \Phi_0 = \Lambda_0 g_1\,,
  \end{equation*}
  with Dirichlet boundary conditions on $\bpi_T^{-1}(\partial_D M)$ and Neumann boundary conditions on $\bpi_T^{-1}(\partial_N M)$.
  Note that given $\omega_1 \in \mathcal H^1_G(M_1)$ we can find $\omega \in \mathcal H^1_G(M)$ such that $\bpi_T^*(\omega) = \omega_1$.
 Indeed, since $\dim(\mathcal H^1_G(M) ) = \dim( \mathcal H^1_G(M_1)) = k$ and $\bpi_T^* \colon \mathcal H^1_G(M) \to \mathcal H^1_G(M_1)$ is injective linear map, it must be an isomorphism.

  Now using~\eqref{e:Lambda1EqLambda} we observe that up to an addition of a scalar multiple of $\Phi_0$, we have
  \begin{equation*}
    g_{1} = \frac{\bpi_{T}^{*}g}{|G_T|^{1/2}} \,,
  \end{equation*}
  where $g = g_\omega$ is defined in~\eqref{e:gomega}.
  Thus, using~\eqref{e:Lambda1EqLambda} again we see
  \begin{align}
   \mathcal{I}_1(\omega_1)
     & =8\pi^{2}|G_T|\int_{M}|\omega|^{2}\frac{\phi_{0}^{2}}{|G_T|}+8\pi|G_T|\int_{M}\frac{\phi_{0}}{|G_T|^{1/2}}\omega\cdot\nabla\paren[\Big]{\frac{g}{|G_T|^{1/2}}}\nonumber \\
   & =8\pi^{2}\int_{M}|\omega|^{2}\phi_{0}^{2}+8\pi\int_{M}\phi_{0} \, \omega\cdot\nabla g
    = \mathcal I(\omega)\,.\label{e:IonM}
  \end{align}

  Since the deck transformation group of $\hat M$ as a cover of $M_1$ is torsion free, we may apply Theorem~\ref{t:hker} to $M_1$.
  Thus, we have
  \begin{equation}\label{e:hkerTorsion}
    \lim_{t\to \infty}
      \paren[\Big]{
	t^{k/2}e^{\Lambda_{0}t}\hat{H}(t,x,y)
	- C_{\mathcal{I}_1}(x,y)
	  \exp\paren[\Big]{-\frac{2\pi^{2}d_{\mathcal{I}_1}^{2}(x,y)}{t}}}
  \end{equation}
  uniformly on $\hat M$.
  Using~\eqref{e:IonM} we see $d_{\mathcal I_1} = d_\mathcal I$.
  Using~\eqref{e:Lambda1EqLambda} and \eqref{e:IonM} we see
  \begin{equation*}
    C_{\mathcal I_1}(x, y) = \frac{1}{\abs{G_T}} C_{\mathcal I}(x, y)\,,
  \end{equation*}
  and inserting this into~\eqref{e:hkerTorsion} finishes the proof.
\end{proof}

The rest of this section is devoted to proving Lemma~\ref{l:minLambdaChi} and Lemma~\ref{l:muBound}.

\subsection{Minimizing the Principal Eigenvalue (Proof of Lemma~\ref{l:minLambdaChi}).}\label{s:lambdamin}

Our aim in this subsection is to prove Lemma~\ref{l:minLambdaChi}, which asserts that the function $\chi \mapsto \lambda_{\chi,0}$ attains a unique global minimum at $\chi = \one$.

If only Neumann boundary condition is imposed on  $\partial M$, Lemma~\ref{l:minLambdaChi} can be proved by adapting the argument in~\cite{Sunada89}.
This yields quantitative upper and lower bounds on the function $\chi\mapsto\lambda_{\chi,0}$ in addition to the global minimum.
Since we only need the global minimum of~$\lambda_{\chi, 0}$, there is a simple proof under Neumann boundary conditions.
We present this first.
We will subsequently provide an independent proof of Lemma~\ref{l:minLambdaChi} under mixed Dirichlet and Neumann boundary conditions.

\begin{proof}[Proof of Lemma~\ref{l:minLambdaChi} under Neumann boundary conditions]
  In this case we know that $\lambda_0 = \lambda_{\one,0}=0$, and the corresponding eigenfunction $s_{\one, 0}$ is constant.
  Thus to prove the lemma it suffices to show that $\lambda_{\chi,0}>0$ for all $\chi\neq\one$.

To see this given $\chi\in\mathcal{G}$ let $s = s_{\chi, 0} \in\mathcal{D}_{\chi}$ be the principal eigenfunction of $-\lap_{\chi}$, and $\lambda = \lambda_{\chi, 0}$ be the principal eigenvalue.
We claim that for any fundamental domain $U \subseteq \hat M$, the eigenvalue~$\lambda$ satisfies
\begin{equation}\label{e:raleigh}
  \lambda \int_{U} \abs{s}^2 \, dx = \int_U \abs{\grad s}^2 \, dx\,.
\end{equation}
Once~\eqref{e:raleigh} is established, one can quickly see that $\lambda > 0$ when $\chi \neq \one$.
Indeed, if $\chi \neq \one$, $s(g(x)) = \chi(g) s(x)$ forces the function $s$ to be non-constant, and now equation~\eqref{e:raleigh} forces $\lambda > 0$.

To prove~\eqref{e:raleigh} observe
\begin{equation}\label{e:stokesThmUg}
  \lambda \int_U \abs{s}^2
    = -\int_{U} \bar{s} \lap_{\chi} s 
    = \int_{U} \abs{\nabla s}^{2} 
      -\int_{\partial U}\bar{s}  \, \partial_\nu s 
      \,.
\end{equation}
Here, $\partial_\nu s  = \nu \cdot  \grad s$ is the outward pointing normal derivative on $\partial U$.
We will show that the twisting condition~\eqref{e:twistingCondition} ensures that the boundary integral above vanishes.

Decompose $\partial U$ as
\begin{equation*}
\partial U = \Gamma_{1}\cup\Gamma_{2} \,,
\quad\text{where }
\Gamma_{1} \defeq \partial U \cap \partial \hat M,
\quad\text{and }
\Gamma_{2} \defeq \partial U - \Gamma_1\,.
\end{equation*}
Note $\Gamma_1$ is the portion of $\partial U$ contained in $\partial\hat{M}$, and  $\Gamma_2$ is the portion of $\partial U$ that is common to neighboring fundamental domains.
Clearly, the Neumann boundary condition~\eqref{e:sOmegaNeumann} implies
\begin{equation*}
\int_{\Gamma_{1}}\bar{s} \, \partial_\nu s
  =0 \,.
\end{equation*}

For the integral over $\Gamma_2$, let $(e_1, \dots, e_k)$ be a basis of $G$ and note that $\Gamma_2$ can be expressed as the disjoint union
\begin{equation*}
\Gamma_{2}=\bigcup_{j=1}^{k}\paren[\big]{\Gamma_{2,j}^{+}\cup\Gamma_{2,j}^{-}}\,,
\end{equation*}
where the $\Gamma_{2,j}^{\pm}$ are chosen so that $\Gamma_{2,j}^+ =  e_{j}( \Gamma_{2,j}^{-})$.
Using the twisting condition~\eqref{e:twistingCondition} and the fact that the action of $e_j$ reverses the direction of the unit normal on $\Gamma_{2,j}^-$, we see
\begin{align*}
  \int_{\Gamma_{2,j}^{+}}\overline{s(x)} \, \partial_\nu s (x) \,dx
    & =-\int_{\Gamma_{2,j}^{-}}
      \overline{s\paren[\big]{e_{j}(y)}} \, \partial_\nu s\paren[\big]{e_{j}(y)} \, dy
    \\
    & =-\int_{\Gamma_{2,j}^{-}} \overline{\chi(e_{j})} \chi(e_{j})
      \, \overline{s(y)} \, \paren[\big]{\partial_\nu s (y)} \, dy
    \\
    &
    =-\int_{\Gamma_{2,j}^{-}}\overline{s(y)} \, \partial_\nu s(y) \, dy \,,
\end{align*}
Consequently,
\begin{equation*}
\int_{\Gamma_{2}}\overline{s} \, \partial_\nu s
  = \sum_{j=1}^{k}\paren[\Big]{ \int_{\Gamma_{2,j}^{+}}+\int_{\Gamma_{2,j}^{-}}}\overline{s} \, {\partial_\nu s}
  =0 \,.
\end{equation*}
and hence the boundary integral in~\eqref{e:stokesThmUg} vanishes.
Thus~\eqref{e:raleigh} holds, and the proof is complete.
\end{proof}

In the general case when $\partial_D M \neq \emptyset$, $\lambda_{\chi,0}>0$ for every $\chi\in\mathcal{G}$, and all eigenfunctions are non-constant.
This causes the previous argument to break down and the proof involves a different idea.
Before beginning the proof, we first make use of a canonical section to transfer the problem to the linear space $\mathcal{H}^1_G$.

Let $\Omega$ be the space of $\C$-valued smooth functions $f\colon M \to \C$ such that $f = 0$ on $\partial_D M$ and $\ip{\grad f, \nu} = 0$ on $\partial_N M$.
Let $\hat f = f \circ \bpi \colon \hat M \to \C$.
Now given $\omega \in \mathcal H^1_G$, let $\sigma_\omega$ (defined in~\eqref{e:sigmaOmegaDef}) be the canonical section and $\chi_\omega \in \mathcal G$ be the exponential as defined in~\eqref{e:expmap}.
Notice that the function $\sigma_\omega \hat f \in \mathcal D_{\chi_\omega}$ is a section on $E_{\chi_\omega}$.
Clearly $\sigma_\omega \hat f = 0$ on $\partial_D \hat M$.
Moreover, since $\omega \cdot \nu = 0$ on $\partial M$ we have
\begin{equation}\label{e:sOmegaNeumann}
  \nu \cdot \grad \sigma_\omega = 0 \qquad\text{on } \partial \hat M\,.
\end{equation}
and hence $\nu \cdot \grad (\sigma_\omega \hat f) = 0$ on $\partial_N \hat M$.
Thus $\sigma_\omega \hat f \in \mathcal D_{\chi_\omega}$, where $\mathcal D_{\chi_\omega}$ is defined in equation~\eqref{e:dchi}, and the map $f\mapsto\hat{f}\sigma_{\omega}$ defines a unitary isomorphism between $\Omega\subseteq L^{2}(M)$ and $\mathcal{D}_{\chi_{\omega}}\subseteq L^{2}(E_{\chi_{\omega}})$ respecting the imposed boundary conditions.

Now, since $\omega$ and $\hat \omega \defeq \omega \circ \bpi$ are both harmonic, we compute 
\begin{equation*}
  \lap_{\chi_{\omega}}(\hat{f}\sigma_{\omega})
    = \paren{ \paren{H_{\omega}f} \circ \bpi } \, \sigma_{\omega} \,,
\end{equation*}
where
$H_{\omega}$ is the self-adjoint operator on $\Omega\subseteq L^{2}(M)$ defined by 
\begin{equation}\label{e:hOmegaDef}
  H_{\omega}f\defeq\lap f+4\pi i \, \omega\cdot\nabla f-4\pi^{2}|\omega|^{2}f \,.
\end{equation}
Here we used the Riemannian metric to identify the $1$-form $\omega$ with a vector field.

The above shows that $\lap_{\chi_{\omega}}$ is unitarily equivalent to $H_{\omega}$.
In particular, eigenvalues of $-H_{\omega}$, denoted by $\mu_{\omega, j}$ are exactly $\lambda_{\chi_\omega, j}$, the eigenvalues of $-\lap_{\chi_\omega}$.
Moreover, the corresponding eigenfunctions, denoted by $\phi_{\omega, j}$, are given by
\begin{equation}\label{e:phiOmegaJ}
\phi_{\omega,j}=\frac{s_{\chi_{\omega},j}}{\sigma_{\omega}} \,,
\quad j\geq0 \,.
\end{equation}
Note that $\phi_{\omega,j}$ is a well-defined function on $M$ that satisfies Dirichlet boundary conditions on $\partial_D M$ and Neumann boundary conditions on $\partial_N M$.

We will now prove the general case of Lemma~\ref{l:minLambdaChi} by minimizing eigenvalues of the operator $-H_\omega$.

\begin{proof}[Proof of Lemma~\ref{l:minLambdaChi}]
  Let $\omega\in \mathcal H^1_G$ and let $\chi_{\omega} = \exp(\omega)\in\mathcal{G}$ be the corresponding representation defined by~\eqref{e:expmap}.
  Let $\mu_{\omega} = \mu_{\omega, 0} = \lambda_{\chi_{\omega},0}$ and $\phi_{\omega} =  \phi_{\omega, 0}$ where $\phi_{\omega, 0}$ is the principal eigenfunction of $-H_\omega$ as defined in~\eqref{e:phiOmegaJ} above.
  Using~\eqref{e:hOmegaDef} we see
  \begin{gather}
    \label{e:phiOmegaDef}
      -\lap\phi_{\omega}-4\pi i\omega\cdot\nabla\phi_{\omega}+4\pi^{2}|\omega|^{2}\phi_{\omega}=\mu_{\omega}\phi_{\omega}\,,
    \\
    \label{e:phiOmega0Def}
    -\lap\phi_{0}=\mu_{0}\phi_{0}\,,
  \end{gather}
  with Dirichlet boundary conditions on $\partial_D \hat M$ and Neumann boundary conditions on $\partial_N \hat M$.
  Here $\mu_0$ and $\phi_0$ denote the principal eigenvalue and eigenfunction respectively when $\omega \equiv 0$.
  Note that when $\omega \in \mathcal H^1_\Z$, the corresponding representation $\chi_\omega$ is the trivial representation $\one$.
  We will show that $\mu_\omega$ above achieves a global minimum precisely when $\omega \in \mathcal H^1_\Z$ and $\chi_\omega = \one$.

  Now let $\epsilon > 0$ and write
  \begin{equation*}
  \overline{\phi_{\omega}} = (\phi_{0}+\epsilon) f
  \quad\text{where }
  f\defeq \frac{\overline{\phi_{\omega}}}{\phi_{0}+\epsilon}\,.
  \end{equation*}
  Multiplying both sides of~\eqref{e:phiOmegaDef} by $\overline{\phi_{\omega}} = (\phi_0 + \epsilon) f$ and integrating over $M$ gives
  \begin{align*}
    -\int_{M}(\lap\phi_{\omega})(\phi_{0}+\epsilon)f & =\int_{M}\nabla\phi_{\omega}\cdot\paren[\big]{(\phi_{0}+\epsilon)\nabla f+f\nabla\phi_{0}} +
    \int_{\partial M} B_1\\
    & =\int_{M}(\phi_{0}+\epsilon)\nabla\phi_{\omega}\cdot\nabla f\\
    & \qquad-\int_{M}\phi_{\omega}\paren[\big]{\nabla f\cdot\nabla\phi_{0}+f\lap\phi_{0}} + \int_{\partial M} B_2\\
    & =\int_{M}\paren[\big]{(\phi_{0}+\epsilon)\nabla\phi_{\omega}-\phi_{\omega}\nabla\phi_{0}}\cdot\nabla f\\
    & \qquad \mathbin{+} \mu_{0}\int_{M}f\phi_{0}\phi_{\omega}+ \int_{\partial M} B_2\,,
  \end{align*}
  where $B_i \colon \partial M \to \C$ are boundary functions that will be combined and written explicitly below (equation~\eqref{e:boundaryTerms}).
  (We clarify that even though the functions above are $\C$-valued, the notation $\grad \phi_\omega \cdot \grad f$ denotes $\sum_i \partial_i \phi_\omega \partial_i f$, and not the complex inner product.)

  Similarly, using the fact that $\omega$ is harmonic, we have
  \begin{align*}
    \MoveEqLeft
    -4\pi i\int_{M}(\phi_{0}+\epsilon) f \omega\cdot\nabla\phi_{\omega}
    \\
      & =-2\pi i\int_{M}(\phi_{0}+\epsilon)f\nabla\phi_{\omega}\cdot\omega
    \\
	& \qquad + 2\pi i\int_{M}\phi_{\omega}\paren[\big]{(\phi_{0}+\epsilon)\nabla f+f\nabla\phi_{0}}\cdot\omega+ \int_{\partial M} B_3
    \\
      & =-2\pi i\int_{M}\paren[\big]{(\phi_{0}+\epsilon)\nabla\phi_{\omega}-\phi_{\omega}\nabla\phi_{0}}\cdot(f\omega)
    \\
	& \qquad +2\pi i\int_{M}(\phi_{0}+\epsilon)\phi_{\omega}\nabla f\cdot\omega+\int_{\partial M} B_3 \,.
  \end{align*}
  Combining the above, we have
  \begin{multline}\label{e:comparingEigenvalues}
    \mu_{\omega}-\mu_{0} \int_{M}f\phi_{0}\phi_{\omega}
    = \int_{M}\paren[\big]{(\phi_{0}+\epsilon)\nabla\phi_{\omega}-\phi_{\omega}\nabla\phi_{0}}\cdot\paren[\big]{\nabla f-2\pi if\omega}
    \\
    +\int_{M}(\phi_{0}+\epsilon)\phi_{\omega}\paren[\big]{4\pi^{2}|\omega|^{2}f+2\pi i\nabla f\cdot\omega}+ \int_{\partial M} B_0 \,,
  \end{multline}
  where 
  \begin{equation}\label{e:boundaryTerms}
    B_0 =-\overline{\phi_{\omega}} \partial_\nu \phi_{\omega}
      + \phi_{\omega}f \partial_\nu \phi_{0}
      - 2\pi i (\phi_{0}+\epsilon)\phi_{\omega}f\omega\cdot\nu \,.
  \end{equation}
  The boundary conditions imposed ensure that $B_0 = 0$ on both $\partial_D M$ and $\partial_N M$.

  Since $f=\overline{\phi_{\omega}}/(\phi_{0}+\epsilon)$, we have
  \begin{equation*}
    \nabla f=\frac{(\phi_{0}+\epsilon)\nabla\overline{\phi_{\omega}}-\overline{\phi_{\omega}}\nabla\phi_{0}}{(\phi_{0}+\epsilon)^{2}}.
  \end{equation*}
  Substituting this into the right hand side of \eqref{e:comparingEigenvalues},
  we obtain a perfect square:
  \begin{equation}\label{e:comparingEvalsSquare}
    \mu_{\omega}-\mu_{0}\int_{M}f\phi_{0}\phi_{\omega}=\int_{M} \abs[\Big]{ 2\pi\phi_{\omega}\omega-\frac{i\paren{(\phi_{0}+\epsilon)\nabla\phi_{\omega}-\phi_{\omega}\nabla\phi_{0}}}{\phi_{0}+\epsilon}}^{2} \,.
  \end{equation}
  In particular, 
  \begin{equation*}
    \mu_{\omega}-\mu_{0}\int_{M}f\phi_{0}\phi_{\omega}=\mu_{\omega}-\mu_{0}\int_{M}\frac{\phi_{0}}{\phi_{0}+\epsilon}|\phi_{\omega}|^{2}\geq0.
  \end{equation*}
  Sending $\epsilon\to 0$, we obtain $\mu_{\omega}\geq\mu_{0}$, and so the function $\mathcal{G}\ni\chi\mapsto\lambda_{\chi,0}$ attains global minimum at $\chi=\one$.
  \medskip

  To see that $\chi=\one$ is the unique global minimum point, suppose that $\lambda_\chi = \lambda_0$ for some $\chi \in \mathcal G$.
  Writing $\chi = \chi_\omega$ for some $\omega \in \mathcal H^1_G$, this means $\mu_\omega = \mu_0$.
  Fatou's lemma and~\eqref{e:comparingEvalsSquare} imply
  \begin{align*}
    \MoveEqLeft
    \int_{M}\abs[\Big]{2\pi\phi_{\omega}\omega-\frac{i\paren[\big]{\phi_{0}\nabla\phi_{\omega}-\phi_{\omega}\nabla\phi_{0}}}{\phi_{0}}}^{2}\\
    & \leq\liminf_{\epsilon\to 0}\int_{M}\abs[\Big]{2\pi\phi_{\omega}\omega-\frac{i\paren[\big]{(\phi_{0}+\epsilon)\nabla\phi_{\omega}-\phi_{\omega}\nabla\phi_{0}}}{\phi_{0}+\epsilon}}^{2}\\
    & =\mu_{\omega}-\mu_{0} = 0\,,
  \end{align*}
  by assumption.
  Hence
  \begin{equation}\label{e:integrandInPerfectSquare}
    2\pi\phi_{\omega}\omega-\frac{i\paren{\phi_{0}\nabla\phi_{\omega}-\phi_{\omega}\nabla\phi_{0}}}{\phi_{0}}=0
    \quad\text{in}\ M \,.
  \end{equation}

  Since $\phi_{\omega}=s_{\chi ,0}/\sigma_{\omega}$, we compute
  \begin{equation*}
    \nabla\phi_{\omega}=\frac{\sigma_{\omega}\nabla s_{\chi,0}-2\pi i \sigma_{\omega}s_{\chi,0}\omega}{\sigma_{\omega}^{2}}.
  \end{equation*}
  Substituting this into \eqref{e:integrandInPerfectSquare}, we see
  \begin{equation*}
    \phi_{0}\nabla s_{\chi,0}=s_{\chi,0}\nabla\phi_{0},
  \end{equation*}
  which implies that 
  \begin{equation*}
    \nabla\paren[\Big]{\frac{s_{\chi,0}}{\phi_{0}}}=0.
  \end{equation*}
  Therefore, $s_{\chi,0}=c\phi_{0}$ for some non-zero constant $c$.
  However, the twisting conditions~\eqref{e:twistingCondition} for $\phi_0$ and $s_{\chi, 0}$ require
  \begin{equation*}
    \phi_0( g(x) ) = \phi_0(x)
    \qquad\text{and}\qquad
    s_{\chi, 0}( g(x)) = \chi(g) s_{\chi, 0} (x)\,,
  \end{equation*}
  for every $g \in \mathcal G$.
  This is only possible if $\chi(g)=1$ for all $g \in \mathcal G$, showing $\chi$ is the trivial representation~$\one$. 
\end{proof}

\subsection{Positivity of the Hessian (Proof of Lemma~\ref{l:muBound}).}\label{s:mubound}
In this subsection we prove Lemma~\ref{l:muBound}.
The main difficulty is proving positivity, which we postpone to the end.

\begin{proof}[Proof of Lemma~\ref{l:muBound}]
Given $\omega \in \mathcal H^1_G$, define
  \begin{equation*}
    \varphi_t = \phi_{t \omega}
    \qquad\text{and}\qquad
    h_t = \mu_{t\omega}\,,
  \end{equation*}
  where $\phi_{t \omega} = \phi_{t\omega, 0}$ is the principal eigenfunction of $-H_{t\omega}$ (equation~\eqref{e:phiOmegaJ}) and $\mu_{t\omega}$ is the corresponding principal eigenvalue.
  We claim that
  \begin{equation}\label{e:hder}
    h'_0 = 0\,,
    \quad
    h''_0 = \mathcal I(\omega)
    \quad\text{and}\quad
    \operatorname{Re}\paren{\varphi'_0} = 0\,,
  \end{equation}
  where $h'$, $\varphi'$ denote the derivatives of $h$ and $\varphi$ respectively with respect to~$t$. This will immediately imply that at $\omega=0$ the quadratic form induced by the Hessian of the map $\omega\mapsto\mu_\omega$ is precisely $\mathcal{I}(\omega)$, hence proving~\eqref{e:muBound} in the lemma.
  
  To  establish \eqref{e:hder}, we first note that~\eqref{e:phiOmegaDef} implies
\begin{equation}\label{e:varphit}
-\lap\varphi_{t}-4\pi it\omega\cdot\nabla\varphi_{t}+4\pi^{2}t^{2}|\omega|^{2}\varphi_{t}=h_t\varphi_{t} \,.
\end{equation}
Conjugating both sides of~\eqref{e:varphit} gives
\begin{equation}\label{e:varphiBar}
-\lap\overline{\varphi_{t}}-4\pi i(-t)\omega\cdot\nabla\overline{\varphi_{t}}+4\pi^{2}(-t)^{2}|\omega|^{2}\overline{\varphi_{t}}=h_t\overline{\varphi_{t}} \,.
\end{equation}
In other words, $\overline{\varphi_{t}}$ is an eigenfunction of $-H_{-t\omega}$
with eigenvalue $h_t$.
Since $h_t = \mu_{t\omega}$ is the principal eigenvalue, this implies $h_{-t}\leq h_t$.
By symmetry, we see that $h_{-t}=h_t$, and hence $h'_0=0$.

To see that $\varphi_{0}'$ is purely imaginary, recall $h_t$ is a simple eigenvalue of $-H_{t\omega}$ when $t$ is small.
Thus
\begin{equation}\label{e:phiPrime0imaginary}
\overline{\varphi_{t}}=\zeta_{t}\varphi_{-t} \,,
\end{equation}
for some $S^1$ valued function $\zeta_{t}$,  defined for small $t$.
Changing $t$ to $-t$, we get 
\begin{equation*}
\overline{\varphi_{-t}}=\zeta_{-t}\varphi_{t}=\zeta_{-t}\overline{\zeta_{t}}\overline{\varphi_{-t}} \,.
\end{equation*}
Therefore, $\zeta_{-t}\overline{\zeta_{t}}=1$, which implies that
$\zeta_{-t}=\zeta_{t}$. In particular, $\zeta'_{0}=0$. Differentiating
\eqref{e:phiPrime0imaginary} and using the fact that
$\zeta_{0}=1$, we get 
\begin{equation*}
\overline{\varphi_{0}'}=-\varphi'_{0} \,,
\end{equation*}
showing that $\varphi_{0}'$ is purely imaginary as claimed.

To compute $h''_0$, we differentiate~\eqref{e:varphit} twice with respect
to $t$.
At $t=0$ this gives
\begin{equation}\label{e:phi0prime}
-\lap\varphi_{0}'-4\pi i\omega\cdot\nabla\varphi_{0}=\lambda_{0}\varphi_{0}',
\end{equation}
and
\begin{equation}\label{e:phi0doublePrime}
-\lap\varphi_{0}''-8\pi i\omega\cdot\nabla\varphi_{0}'+8\pi^{2}|\omega|^{2}\phi_{0}=h''_{0}\phi_{0}+\lambda_{0}\varphi_{0}'' \,,
\end{equation}
since $\varphi_0 = \phi_0$.
Multiplying both sides of \eqref{e:phi0doublePrime} by $\phi_{0}$ and integrating over~$M$ gives
\begin{equation}\label{e:hDouplePrime0}
h_{0}''=\int_{M}\paren[\big]{8\pi^{2}|\omega|^{2}\phi_{0}^{2}-8\pi i\phi_{0}\omega\cdot\nabla\varphi_{0}'} \,.
\end{equation}

Recalling that~$\varphi_0'$ is purely imaginary, we let $g_\omega$ be the real valued function defined by $g_\omega = -i \varphi'_0$.
Now equation~\eqref{e:phi0prime} shows that~$g_\omega$ satisfies~\eqref{e:gomega}.
Moreover since~$\varphi_0 = 0$ on $\partial_D M$ and $\nu \cdot \grad \varphi_0 = 0$ on $\partial_N M$, the function~$g_\omega$ satisfies the boundary conditions~\eqref{e:gomegaBC}.
Therefore,~\eqref{e:hDouplePrime0} reduces to~\eqref{e:Idef}, showing that $h''_0 = \mathcal I(\omega)$ as claimed.
\medskip

Finally, we show that $\omega\mapsto\mathcal{I}(\omega)$ defined by~\eqref{e:Idef} is a well defined positive definite quadratic form on $\mathcal{H}^1_G$.  
 To see that $\mathcal I$ is well defined, we first note that in order for~\eqref{e:gomega} to have a solution, we need to verify the solvability condition
  \begin{equation*}
    \int_M \phi_0 \paren[\big]{ 4 \pi \omega \cdot \grad \phi_0 }  = 0\,.
  \end{equation*}
  This is easily verified as 
  \begin{equation}\label{e:intphi0DotStuff}
    \int_{M}\phi_{0}\omega\cdot\nabla\phi_{0}=\frac{1}{2}\int_{M}\omega\cdot\nabla\phi_{0}^{2}=0 \,.
  \end{equation}
  Hence $g_\omega$ is uniquely defined up to the addition of a scalar multiple of $\phi_0$ (the kernel of $\lap + \lambda_0$).
  Now, using~\eqref{e:intphi0DotStuff} again, we see that replacing $g_\omega$ with $g_\omega + \alpha \phi_0$ does not change the value of $\mathcal I(\omega)$.
  Thus, $\mathcal I(\omega)$ is a well defined function.
  The fact that $\mathcal I$ is a quadratic form~\eqref{e:Idef} and the fact that
  \begin{equation*}
    g_{\tau+\omega} = g_{\tau}+g_{\omega}\quad \pmod{\phi_{0}} \,.
  \end{equation*}

  It remains to show that~$\mathcal I$ is positive definite.  Note that, in view of Lemma~\ref{l:minLambdaChi}, we already know that $\mathcal I$ induces a positive \emph{semi}-definite quadratic form on $\mathcal H^1_G$.

  For the convenience of notation, let $g = g_\omega = -i \varphi_0'$ as above.
  As before we write
  \begin{equation*}
    g = (\phi_0 + \epsilon) f_\epsilon\,,
    \quad\text{where }
    f_\epsilon \defeq \frac{g}{\phi_0 + \epsilon}\,,
  \end{equation*}
  and will multiplying both sides of~\eqref{e:gomega} by $(\phi_0 + \epsilon) f_\epsilon$ and integrating.
  In preparation for this we compute
  \begin{multline*}
    -\int_{M}(\phi_{0}+\epsilon)f_{\epsilon}\lap g
      =\int_{M}\nabla g\cdot\paren[\Big]{f_{\epsilon}\nabla\phi_{0}+(\phi_{0}+\epsilon)\nabla f_{\epsilon}}
    \\
      =\lambda_{0}\int_{M}\phi_{0}f_{\epsilon}g-\int_{M}g\nabla f_{\epsilon}\cdot\nabla\phi_{0}+\int_{M}(\phi_{0}+\epsilon)\nabla f_{\epsilon}\cdot\nabla g \,,
  \end{multline*}
  and 
  \begin{align*}
    4\pi\int_{M}(\phi_{0}+\epsilon)f_{\epsilon}\omega\cdot\nabla(\phi_{0}+\epsilon) & =2\pi\int_{M}f_{\epsilon}\omega\cdot\nabla(\phi_{0}+\epsilon)^{2}\\
    & =-2\pi\int_{M}(\phi_{0}+\epsilon)^{2}\nabla f_{\epsilon}\cdot\omega \,.
  \end{align*}
  We remark that when integrating by parts above, the boundary terms that arise all vanish because of the boundary conditions imposed.
  Thus, multiplying~\eqref{e:gomega} by $(\phi_0 + \epsilon) f_\epsilon$ and integrating gives
  \begin{align}
    \nonumber
    \lambda_{0}\int_{M}g^{2}\paren[\Big]{1-\frac{\phi_{0}}{\phi_{0}+\epsilon}}&=  \int_{M}(\phi_{0}+\epsilon)\nabla f_{\epsilon}\cdot\nabla g-\int_{M}g\nabla f_{\epsilon}\cdot\nabla(\phi_{0}+\epsilon)
    \\
    \label{eq: after multiplying g on both sides}
    & \qquad \mathbin{+} 2\pi\int_{M}(\phi_{0}+\epsilon)^{2}\nabla f_{\epsilon}\cdot\omega \,.
  \end{align}

  Writing $\tau\defeq 2\pi\omega$ and adding the integral 
  \begin{align*}
    J_{\epsilon} & \defeq\int_{M}(\phi_{0}+\epsilon)\tau\cdot\nabla g-\int_{M}g\tau\cdot\nabla(\phi_{0}+\epsilon)
    +\int_{M}(\phi_{0}+\epsilon)^{2}|\tau|^{2}
  \end{align*}
  to both sides of \eqref{eq: after multiplying g on both sides}, we obtain
  \begin{multline}\label{eq: after adding J_epsilon on both sides}
    J_{\epsilon}+\lambda_{0}\int_{M}g^{2}\paren[\Big]{1-\frac{\phi_{0}}{\phi_{0}+\epsilon}}
      = \int_{M}(\phi_{0}+\epsilon)(\nabla f_{\epsilon}+\tau)\cdot\nabla g
    \\
      -\int_{M}g(\nabla f_{\epsilon}+\tau)\cdot\nabla(\phi_{0}+\epsilon)
      +\int_{M}(\phi_{0}+\epsilon)^{2}(\nabla f_{\epsilon}+\tau)\cdot\tau \,.
  \end{multline}
  Now, since $g=(\phi_{0}+\epsilon)f_{\epsilon}$, we compute
  \begin{equation*}
    \nabla g=f_{\epsilon}\nabla(\phi_{0}+\epsilon)+(\phi_{0}+\epsilon)\nabla f_{\epsilon} \,.
  \end{equation*}
  Substituting this into \eqref{eq: after adding J_epsilon on both sides} gives
  \begin{equation}\label{e:positivityOfI1}
    J_{\epsilon}+\lambda_{0}\int_{M}g^{2}\paren[\Big]{1-\frac{\phi_{0}}{\phi_{0}+\epsilon}}=\int_{M}(\phi_{0}+\epsilon)^{2}|\nabla f_{\epsilon}+\tau|^{2}\geq0 \,.
  \end{equation}

  Using~\eqref{e:Idef} we see
  \begin{equation}
    \mathcal{I}(\omega)
    = 8\pi^{2}\int_{M}|\omega|^{2}\phi_{0}^{2}+4\pi\int_{M}\phi_{0}\omega\cdot\nabla g -4\pi\int_{M}g\omega\cdot\nabla\phi_{0} \,,
  \end{equation}
  and hence it follows that
  \begin{equation*}
    \lim_{\epsilon\to 0}J_{\epsilon}=\frac{1}{2}\mathcal{I}(\omega) \,.
  \end{equation*}
  Also by the dominated convergence theorem, the second term on the left hand side of \eqref{e:positivityOfI1} goes to zero as $\epsilon\to 0$.
  This shows $\mathcal{I}(\omega)\geq0$.

  It remains to show $\mathcal I(\omega) > 0$ if $\omega \neq 0$.
  Note that if $\mathcal{I}(\omega)=0$, then Fatou's lemma and~\eqref{e:positivityOfI1} imply
  \begin{equation*}
    \int_{M}\phi_{0}^{2}|\nabla f+\tau|^{2}\leq\liminf_{\epsilon\downarrow0}\paren[\Big]{J_{\epsilon}+\lambda_{0}\int_{M}g^{2}\paren[\Big]{1-\frac{\phi_{0}}{\phi_{0}+\epsilon}}}=0 \,,
  \end{equation*}
  where $f\defeq g/\phi_{0}$.
  Therefore $\grad f + \tau = 0$ in $M$ and hence $\omega = - \grad f / (2\pi)$.
  Since $\omega \in \mathcal H^1_G \subseteq \mathcal H^1$, this forces
  \begin{equation*}
    \lap f = 0 \quad\text{in }M\,,
    \qquad\text{and}\qquad
    \nu \cdot \grad f = 0 \quad\text{on } \partial M\,.
  \end{equation*}
  Consequently $\grad f = 0$, which in turn implies~$\omega = 0$. This completes the proof of the positivity of $\mathcal{I}$.
\end{proof}

\section{Proof of the Winding Number Asymptotics (Theorem~\ref{t:winding}).}\label{s:pfwinding}

In this section, we study the long time behaviour of the abelianized winding number of reflected Brownian motion on a manifold~$M$.
We begin by using Theorem~\ref{t:hker} to prove Theorem~\ref{t:winding} (Section~\ref{s:windingProof}).
Next, in Section~\ref{s:tobyWerner} we discuss the connection of our results with those obtained by Toby and Werner~\cite{TobyWerner95}.
Finally, in Section~\ref{s:windingDomain}, we outline a direct probabilistic proof of Theorem~\ref{t:winding}.

\subsection{Proof of Theorem~\ref{t:winding}}\label{s:windingProof}
We obtain the long time behaviour of the abelianized winding of reflected Brownian motion in $M$ by applying Theorem~\ref{t:hker} in this context.
Let $\hat{M}$ be a covering space of $M$ with deck transformation group%
\footnote{
  The existence of such a cover is easily established by taking the quotient of the universal cover $\bar{M}$ by the action of the commutator of $\pi_1(M)$.
}
$\pi_1(M)_{\ab}$.
In view of the covering factorization~\eqref{e:torsionFactorization}, we may, without loss of generality, assume that $\tor(\pi_1(M)_{\ab})=\{0\}$.
Note that since the deck transformation group $G = \pi_1(M)_\ab$ by construction, we have $\mathcal H^1_G = \mathcal H^1$.
Given $n \in \Z^k$ ($k=\rm{rank}(G)$), define $g_n \in G$ by
\begin{equation*}
  g_n \defeq \sum_{i=1}^k n_i \pi_G(\gamma_i)\,,
  \quad\text{where } n = (n_1, \dots, n_k) \in \Z^k\,.
\end{equation*}
Here $(\pi_G(\gamma_1), \dots, \pi_G(\gamma_k))$ is the basis of $G$ chosen in Section~\ref{s:winding}.
Clearly $n \mapsto g_n$ is an isomorphism between $G$ and $\Z^k$.

\begin{lemma}\label{l:dIgn}
  For any $x, y \in \hat M$ and $n \in \Z^k$ we have
  \begin{equation*}
    d_\mathcal I(x, g_n(y) )^2 = (A^{-1}n) \cdot n + O(\abs{n})\,.
  \end{equation*}
  Here $A$ is the matrix $(a_{i,j})$ defined by
  \begin{equation}\label{e:cov}
    a_{i,j}
      \defeq \ip{\omega_i, \omega_j}_\mathcal I
      = \frac{8 \pi^2}{\vol(M)} \int_M \omega_i \cdot \omega_j  \,.
  \end{equation}
\end{lemma}
\begin{proof}
  Given $\omega \in \mathcal H^1$ we compute
  \begin{equation}\label{e:xix:gy}
    \xi_{x, g_n (y)}(\omega)
      = \int_{x}^y \bpi^*(\omega) + \int_y^{g_n(y)} \bpi^*(\omega) \,,
  \end{equation}
  where the integrals are performed along any smooth path in~$\hat M$ connecting the endpoints.
  By construction of $\hat M$,  $\mathcal H^1_G = \mathcal H^1$, and hence both integrals above are independent of the path of integration.
  Moreover, the second integral is independent of $y$.
  Hence, if for any $g \in G$ we define $\psi_g\colon \mathcal H^1 \to \R$ by
  \begin{equation*}
    \psi_g(\omega) = \int_y^{g(y)} \bpi^*(\omega) \,,
  \end{equation*}
  then~\eqref{e:xix:gy} becomes
  \begin{equation*}
    \xi_{x, g_n(y)} (\omega) = \xi_{x,y}(\omega) + \psi_{g_n}(\omega)\,.
  \end{equation*}
  From this we compute
  \begin{equation*}
    d_\mathcal I(x,g_n(y))^2
      = d_\mathcal I(x,y)^2
	+ \sum_{i=1}^k n_i \ip{\psi_{\pi_G(\gamma_i)}, \xi_{x,y} }_{\mathcal I^*}
	+ \sum_{i,j=1}^k n_i n_i \ip{\pi_G(\gamma_i), \pi_G(\gamma_j)}_{\mathcal I^*}\,.
  \end{equation*}
  Since $(\omega_1, \dots, \omega_k)$ is the dual basis to $(\pi_G(\gamma_1), \dots, \pi_G(\gamma_j))$, we have
  \begin{equation*}
    \ip{\pi_G(\gamma_i), \pi_G(\gamma_j)}_{\mathcal I^*} = (A^{-1})_{i,j}\,,
  \end{equation*}
  from which the first equality in~\eqref{e:cov} follows.
  The second equality follows from the fact that~\eqref{e:INeumann} holds under Neumann boundary conditions (Remark~\ref{r:neumann}).
\end{proof}

Now we prove Theorem~\ref{t:winding}.

\begin{proof}[Proof of Theorem~\ref{t:winding}]
  Recall in Section~\ref{s:winding} we decomposed the universal cover $\bar M$ as the disjoint union of fundamental domains $\bar U_g$ indexed by $g \in \pi_1(M)$.
  Projecting these domains to the cover $\hat M$ we write $\hat M$ as the disjoint union of fundamental domains $\bar U_g$ indexed by $g \in G$.
  Let $\hat W$ be the lift of the trajectory of~$W$ to $\hat M$, and observe that if $\hat W(t) \in \hat U_{g_n}$, then $\rho(t) = n$.
  
  We use this to compute the characteristic function of $\rho(t) / \sqrt{t}$ as follows.
  Since the generator of $\hat W$ is $\frac{1}{2} \lap$, its transition density is given by $\hat H(t/2, \cdot, \cdot)$.
  Hence, for any $z \in \R^k$ we have
  \begin{multline*}
    \E^x \exp\paren[\Big]{ \frac{i z \cdot \rho(t)}{t^{1/2}} }
      = \sum_{n \in \Z^k}
	  \exp\paren[\Big]{ \frac{i z \cdot n}{t^{1/2}} }
	  \P^x( \hat W(t) \in \hat U_{g_n} )
    \\
      = \sum_{n \in \Z^k} \int_{\hat U_{g_n}}
	  \hat H\paren[\Big]{ \frac{t}{2}, x, y }
	  \exp\paren[\Big]{ \frac{i z \cdot n}{ t^{1/2} } }
	  \, dy\,.
  \end{multline*}
  By Theorem~\ref{t:hker} and Remark~\ref{r:neumann}, this means that uniformly in $x \in \hat M$ we have
  \begin{align*}
    \MoveEqLeft
    \lim_{t\to \infty}
      \E^x \exp\paren[\Big]{ \frac{i z \cdot \rho(t)}{t^{1/2}} }
    \\
      &= C_\mathcal I \lim_{t\to \infty}
	  \sum_{n \in \Z^k} \int_{\hat U_{g_n}}
	      \frac{2^{k/2}}{t^{k/2}}
	      \exp\paren[\Big]{ -\frac{4\pi^2 d_\mathcal I(x, g_n(y))^2 }{t} + \frac{i z \cdot n}{ t^{1/2} } } \, dy
    \\
      &= C_\mathcal I \lim_{t\to \infty}
	  \sum_{n \in \Z^k}
	      \frac{2^{k/2}}{t^{k/2}}
	      \exp\paren[\Big]{ -\frac{4\pi^2 (A^{-1} n) \cdot n}{t} + \frac{i z \cdot n}{ t^{1/2} } }\,.
  \end{align*}
  Here the last equality followed from Lemma~\ref{l:dIgn} above.
  Now the last term is the Riemann sum of a standard Gaussian integral, and hence
  \begin{equation*}
    \lim_{t\to \infty}
      \E^x \exp\paren[\Big]{ \frac{i z \cdot \rho(t)}{t^{1/2}} }
    = 2^{k/2} C_\mathcal I \int_{\zeta \in \R^k} 
      \exp\paren[\Big]{ -4\pi^2 (A^{-1} \zeta) \cdot \zeta + i z \cdot \zeta  } \, d\zeta\,.
  \end{equation*}
  This shows that as $t \to \infty$, $\rho(t) / \sqrt{t}$ converges to a normally distributed random variable with mean $0$ and covariance matrix $A / (8 \pi^2)$.
  By~\eqref{e:sigmadef} and~\eqref{e:cov} we see that $\Sigma = A / (8 \pi^2 )$, which completes the proof of the second assertion in~\eqref{e:rhoLim}.
  The first assertion follows immediately from the second assertion and Chebychev's inequality.
  This completes the proof of Theorem~\ref{t:winding}.
\end{proof}
\subsection{Relation to the Work of Toby and Werner}\label{s:tobyWerner}

Toby and Werner~\cite{TobyWerner95} studied the long time behaviour of the winding of an obliquely reflected Brownian motion in bounded planar domains.
In this case, we describe their result and relate it to Theorem~\ref{t:winding}.

Let $\Omega \subseteq \R^2$ be a bounded domain with $k$ holes $V_1,\cdots,V_k$ of positive volume.
Let $W_{t}$ be a reflected Brownian motion in $\Omega$ with a non-tangential reflecting vector field $u \in C^1(\partial \Omega)$.
Let $p_{1},\cdots,p_{k}$ be $k$ distinct points in $\R^{2}$.
For $1\leq j\leq k$, define $\rho(t,p_{j})$ to be the winding number of $W_{t}$ with respect to the point $p_{j}$. 

\begin{theorem}[Toby and Werner, 1995]\label{t:tobyWerner}
There exist constants $a_i$, $b_i$, depending on the domain $\Omega$, such that 
\begin{equation}\label{e:TWrho}
  \frac{1}{t}\paren[\big]{\rho(t,p_{1}),\cdots,\rho(t,p_{k})}
    \xrightarrow[t \to \infty]{w}
    \paren[\big]{a_{1}C_{1}+b_{1},\cdots,a_{k}C_{k}+b_{k}}\,.
\end{equation}
Here $C_{1}$, \dots, $C_{k}$ are standard Cauchy variables.
Moreover, for any $j$ such that $p_{j}\notin \Omega$, we must have $a_{j}=0$.
\end{theorem}

When $p_{j}\in \Omega$, the process $W$ can wind many times around $p_j$ when it gets close to $p_j$.
This is why the heavy tailed Cauchy distribution arises in Theorem~\ref{t:tobyWerner}, and the limiting process is non-degenerate precisely when each $p_j \in \Omega$.

In the context of Theorem~\ref{t:winding} we require compactness of the domain.
This will only be true when when $p_j \not \in \Omega$ for all $j$, in which case each $a_j = 0$.
We now describe how the constants $b_{j}$ are computed in~\cite{TobyWerner95}.

Recall (see for instance Stroock-Varadhan \cite{StroockVaradhan71}) that reflected Brownian motion has the semi-martingale representation
\begin{equation}\label{e:smgRepBM}
  W_{t}=\beta_{t}+\int_{0}^{t}u(W_{s}) \, dL_{s}\,.
\end{equation}
Here $\beta_{t}$ is a two dimensional Brownian motion, $u$ is the reflecting vector field on $\partial \Omega$, and $L_{t}$ is a continuous increasing process which increases only when $W_{t}\in\partial \Omega$.
We also know that the process $W_{t}$ has a unique invariant measure, which we denote by~$\mu$.
Now, the constants $b_{j}$ are given by 
\begin{equation}\label{e:TWb}
  b_{j}
    = \frac{1}{2\pi}\int_{p \in \Omega}
	\E^{p}\brak[\Big]{\int_{0}^{1}u_{j}(W_{s})dL_{s}} \, d\mu(p) \,,
\end{equation}
where $u_{j} \colon \partial \Omega \to \R$ is defined by 
\begin{equation*}
u_{j}(p)\defeq\frac{u(p) \cdot (p-p_{j})^\perp }{\abs{p-p_{j}}} \,.
\end{equation*}
Above the notation $\perp$ denotes the rotation of a point counter clockwise by an angle of $\pi/2$ about the origin.
That is, if $q = (q_1, q_2)  \in \R^2$, then $q^\perp = (-q_2, q_1)$.

In the case that the reflection is normal, we claim that each $b_j = 0$.
\begin{proposition}\label{p:TWb}
Suppose $W_{t}$ is the normally reflected Brownian motion in $\Omega$, and $p_{j}\in V_{j}$ for each $j$.
Then $b_{j}=0$ for all $j$, and consequently
\begin{equation*}
\lim_{t\rightarrow\infty}\frac{\rho(t, p_{j})}{t} \xrightarrow[t \to \infty]{p} 0\,.
\end{equation*}
\end{proposition}

Note that Proposition~\ref{p:TWb} is simply the first assertion in~\eqref{e:rhoLim}, and follows trivially from the second assertion (the central limit theorem).
For completeness, we provide an independent proof of Proposition~\ref{p:TWb} directly using~\eqref{e:TWb}.

\begin{proof}[Direct proof of Proposition~\ref{p:TWb}]
Fix $1\leq j\leq k$. Let $w(t,p)$ be the solution to the
following initial-boundary value problem:
\begin{equation}\label{e:w}
  \left\{
    \begin{alignedat}{2}
      \span
	\partial_t w -\frac{1}{2}\lap w = 0
	  &\qquad& \text{in } (0,\infty)\times \Omega\,,
      \\
      \span
	\nu \cdot \grad w =-u_{j}
	  && \text{on } (0,\infty)\times\partial \Omega\,,
      \\
      \span
	\lim_{t\downarrow0}w(t,\cdot)=0
	  && \text{in}\ \Omega \,,
    \end{alignedat}
  \right.
\end{equation}
where $\nu$ is the outward pointing unit normal on the boundary.
By applying It\^o's formula to the process $[0,t-\epsilon]\ni s\mapsto w(t-s,W_{s})$ and using the semi-martingale representation \eqref{e:smgRepBM} of $W_{t}$, we get
\begin{align*}
w(t,p)-\E^{p}\brak[\Big]{w(\epsilon,W_{t-\epsilon})} & =-\E^{p}\brak[\Big]{\int_{0}^{t-\epsilon} \nu \cdot \grad w (W_{s},t-s)dL_{s}}\\
 & =\E^{p}\brak[\Big]{\int_{0}^{t-\epsilon}u_{j}(W_{s})dL_{s}},
\end{align*}
where in the last identity we have used the fact that $dL_{s}$ is
carried by the set $\{s\geq0:W_{s}\in\partial \Omega\}$. Since $\P(B_{t}\in\partial U)=0$,
sending $\epsilon\downarrow0$ and using the dominated convergence theorem gives
\begin{equation*}
  w(t,p)=\E^{p}\brak[\Big]{\int_{0}^{t}u_{j}(W_{s})dL_{s}} \,.
\end{equation*}

On the other hand, according to Harrison, Landau and Shepp \cite{HarrisonLandauEA85},
Theorem 2.8, the invariant measure $\mu$ of $W_{t}$ is the unique
probability measure on the closure $\bar{\Omega}$ of $\Omega$ that $\mu(\partial \Omega)=0$
and 
\begin{equation*}
  \int_{\Omega}\lap f(p) \, d\mu(p) \leq 0
  \quad
  \text{for all $f\in C^{2}(\bar{\Omega})$ with $\nu \cdot \grad f \leq0$ on $\partial \Omega$.}
\end{equation*}
Stokes' theorem now implies $\mu$ is the normalized Lebesgue measure on $\Omega$.
Consequently,
\begin{equation*}
b_{j}
  =\frac{1}{2\pi\vol (\Omega)}\int_{\Omega}\E^{p}\brak[\Big]{\int_{0}^{1}u_{j}(W_{s}) \, dL_{s}} \, dp
  =\frac{1}{2\pi\vol (\Omega)}\int_{\Omega}w(1,p) \, dp \,.
\end{equation*}

Integrating~\eqref{e:w} over $\Omega$ and using the boundary conditions yields
\begin{align*}
0 & = \partial_t \int_{\Omega}w \, dp-\int_{\Omega}\lap w \, dp\\
 & =\partial_t \int_{\Omega}w \, dp+\int_{\partial \Omega}u_{j}(p) \, dp\\
 & =\partial_t \int_{\Omega}w \, dp-\int_{\partial \Omega} \nu \cdot \frac{(p-p_{j})^\perp}{|p-p_{j}|} \, dp \,.
\end{align*}
Since when $p_{j}\in V_{j}$ the vector field $p\mapsto\sfrac{(p-p_{j})^\perp}{|p-p_{j}|}$  is a divergence free vector field on $\bar{\Omega}$, the last integral above above vanishes.
Thus
\begin{equation*}
 \partial_t \int_{\Omega}w \, dp=0\,,
\end{equation*}
and since $w=0$ when $t=0$, $w = 0$ for all $t \geq 0$, and hence $b_{j}=0$.
\end{proof}

Therefore, in the case with normal reflection and $p_j\in V_j$, the result of Toby and Werner becomes a law of large numbers and Theorem~\ref{t:winding} provides a central limit theorem. In this case, our result is a refinement of Theorem~\ref{t:tobyWerner}.

\begin{remark}
  The setting of Toby and Werner~\cite{TobyWerner95} is more general.
  Namely they study obliquely reflected Brownian motion, and the case of punctured domains (i.e.\ when $z_{j}\in \Omega$) where the limiting behavior is the (heavy tailed) Cauchy distribution. 
\end{remark}

\subsection{A Direct Probabilistic Proof of Theorem~\ref{t:winding}}\label{s:windingDomain}

As mentioned earlier, Theorem~\ref{t:winding} can also be proved directly by using a probabilistic argument.
The proof is particularly simple in the case of Euclidean domains with smooth boundary.
On manifolds, however, there are a few details that need to be verified.
While these are direct generalizations of their Euclidean counterparts, to our best knowledge, they are not readily available in the literature.

First suppose  $\gamma\colon [0, \infty) \to M$ is a smooth path.
Let $\rho(t, \gamma)$ be the $\Z^k$-valued winding number of $\gamma$, as in Definition~\ref{d:winding}.
Namely, let $\bar \gamma$ be the lift of $\gamma$ to the universal cover of $M$, and let $\rho(t, \gamma) = (n_1, \dots, n_k)$ if
\begin{equation*}
  \pi_G\paren[\big]{ \bar{\bm{g}}(\bar \gamma(t))} = \sum_{i=1}^k n_i \pi_G(\gamma_i) \,.
\end{equation*}
By our choice of $(\omega_1, \dots, \omega_k)$ we see that $\rho_i(t, \gamma)$, the $i^\text{th}$ component of~$\rho(t, \gamma)$, is precisely the integer part of $\theta_i(t, \gamma)$, where
\begin{equation}\label{e:gamma}
  \theta_i(t, \gamma)
    \defeq \int_{\gamma([0, t])} \omega_i
    = \int_0^t \omega_i(\gamma(s)) \, \gamma'(s) \, ds\,.
\end{equation}
If $M$ is a planar domain with $k$ holes, and the forms $\omega_i$ are chosen as in Remark~\ref{r:planar}, then $2 \pi \theta_i(t, \gamma)$ is the total angle $\gamma$ winds around the $k^\text{th}$ hole up to time $t$.

In the case that $\gamma$ is not smooth, the theory of rough paths can be used to give meaning to the above path integrals.
Moreover, when~$\gamma$ is the trajectory of a reflected Brownian motion on~$M$, we know that the integral obtained via the theory of rough paths agrees with the Stratonovich integral.
To fix notation, let $W$ be a reflected Brownian motion in $M$, and $\rho(t)=(\rho_{1}(t),\cdots,\rho_{k}(t))$ to be the $\mathbb{Z}^{k}$-valued winding number of $W$ as in Definition~\ref{d:winding}.
Then we must have $\rho_{i}(t)=\lfloor\theta_{i}(t)\rfloor$, where $\theta_{i}(t)$ is the Stratonovich integral
\begin{equation}\label{e:thetaiStrat}
  \theta_{i}(t)=\int_{0}^{t}\omega_{i}(W_{s})\circ dW_{s} \,.
\end{equation}

In Euclidean domains, the long time behaviour of this integral can be obtained as follows.
The key point to note is that the forms $\omega_i$ are chosen to be harmonic in~$M$ and tangential on $\partial M$.
Consequently, using the semi-martingale decomposition~\eqref{e:smgRepBM}, we see that $\theta$ is a martingale with quadratic variation given by
\begin{equation}\label{e:qvTheta}
  \langle\theta_{i},\theta_{j}\rangle_{t}=\int_{0}^{t}\omega_{i}(W_{s})\cdot\omega_{j}(W_{s}) \, ds \, .
\end{equation}
Moreover, by Harrison et.\ al.~\cite{HarrisonLandauEA85}, the unique invariant measure of $W_{t}$ is the normalized volume measure.
Thus, by the ergodic theorem, 
\begin{equation*}
  \lim_{t\rightarrow\infty}\frac{1}{t}\langle\theta_{i},\theta_{j}\rangle_{t}=\frac{1}{{\rm vol}(M)}\int_{M}\omega_{i}\cdot\omega_{j}
\end{equation*}
for almost surely.
Now, the martingale central limit theorem~\cite[Theorem 3.33 and Corollary 3.34]{PavliotisStuart08}, implies
conclude that 
\begin{equation*}
  \frac{\theta_t}{\sqrt{t}} \xrightarrow[w]{t \to \infty} \mathcal N(0, \Sigma)\,,
\end{equation*}
where the covariance matrix $\Sigma$ is given by~\eqref{e:sigmadef}. 
\medskip

In order for the above argument to work on compact Riemannian manifolds, one needs to establish a few of the results used above in this setting.
First, one needs to show the analogue of the semi-martingale decomposition~\eqref{e:smgRepBM} on manifolds with boundary.
While this is a straightforward adaptation of~\cite{StroockVaradhan71}, there is (to the best of our knowledge) no easily available reference.
Next, one needs to use the fact $\omega \in \mathcal H^1$ to show that~$\theta_i$ is a martingale with quadratic variation~\eqref{e:qvTheta}.
This can be done by breaking the Stratonovich integral defining~$\theta_i$ (equation~\eqref{e:thetaiStrat}) into pieces that are entirely contained in local coordinate charts, and using the analogue of~\eqref{e:smgRepBM} together with the fact that $\omega \in \mathcal H^1$.
Now the rest of the proof is the same as in the case of Euclidean domains. 

\section*{Acknowledgements.}
The authors wish to thank Jean-Luc Thiffeault for suggesting this problem to us and many helpful discussions.

\bibliographystyle{halpha-abbrv}
\bibliography{refs}

\begin{thebibliography}{BGV92}
\expandafter\ifx\csname url\endcsname\relax
  \def\url#1{\texttt{#1}}\fi
\expandafter\ifx\csname doi\endcsname\relax
  \def\doi#1{\burlalt{doi:#1}{http://dx.doi.org/#1}}\fi
\expandafter\ifx\csname urlprefix\endcsname\relax\def\urlprefix{URL }\fi
\expandafter\ifx\csname href\endcsname\relax
  \def\href#1#2{#2}\fi
\expandafter\ifx\csname burlalt\endcsname\relax
  \def\burlalt#1#2{\href{#2}{#1}}\fi

\bibitem[AJ01]{AnkerJi01}
J.-P. Anker and L.~Ji.
\newblock Heat kernel and {G}reen function estimates on noncompact symmetric
  spaces. {II}.
\newblock In {\em Topics in probability and {L}ie groups: boundary theory},
  volume~28 of {\em CRM Proc. Lecture Notes}, pages 1--9. Amer. Math. Soc.,
  Providence, RI, 2001.

\bibitem[BGV92]{BerlineGetzlerEA92}
N.~Berline, E.~Getzler, and M.~Vergne.
\newblock {\em Heat kernels and {D}irac operators}, volume 298 of {\em
  Grundlehren der Mathematischen Wissenschaften [Fundamental Principles of
  Mathematical Sciences]}.
\newblock Springer-Verlag, Berlin, 1992.
\newblock \doi{10.1007/978-3-642-58088-8}.

\bibitem[CK91]{ChavelKarp91}
I.~Chavel and L.~Karp.
\newblock Large time behavior of the heat kernel: the parabolic
  {$\lambda$}-potential alternative.
\newblock {\em Comment. Math. Helv.}, 66(4):541--556, 1991.
\newblock \doi{10.1007/BF02566664}.

\bibitem[GK94]{GeigerKersting94}
J.~Geiger and G.~Kersting.
\newblock Winding numbers for {$2$}-dimensional, positive recurrent diffusions.
\newblock {\em Potential Anal.}, 3(2):189--201, 1994.
\newblock \doi{10.1007/BF01053432}.

\bibitem[Gri99]{Grigoryan99}
A.~Grigor'yan.
\newblock Estimates of heat kernels on {R}iemannian manifolds.
\newblock In {\em Spectral theory and geometry ({E}dinburgh, 1998)}, volume 273
  of {\em London Math. Soc. Lecture Note Ser.}, pages 140--225. Cambridge Univ.
  Press, Cambridge, 1999.
\newblock \doi{10.1017/CBO9780511566165.008}.

\bibitem[HLS85]{HarrisonLandauEA85}
J.~M. Harrison, H.~J. Landau, and L.~A. Shepp.
\newblock The stationary distribution of reflected {B}rownian motion in a
  planar region.
\newblock {\em Ann. Probab.}, 13(3):744--757, 1985.
\newblock \urlprefix\url{http://www.jstor.org/stable/2243711}.

\bibitem[KS00]{KotaniSunada00}
M.~Kotani and T.~Sunada.
\newblock Albanese maps and off diagonal long time asymptotics for the heat
  kernel.
\newblock {\em Comm. Math. Phys.}, 209(3):633--670, 2000.
\newblock \doi{10.1007/s002200050033}.

\bibitem[Lal93]{Lalley93}
S.~P. Lalley.
\newblock Finite range random walk on free groups and homogeneous trees.
\newblock {\em Ann. Probab.}, 21(4):2087--2130, 1993.
\newblock \urlprefix\url{http://www.jstor.org/stable/2244712}.

\bibitem[Li86]{Li86}
P.~Li.
\newblock Large time behavior of the heat equation on complete manifolds with
  non-negative {R}icci curvature.
\newblock {\em Annals of Mathematics}, 124(1):1--21, 1986.
\newblock \doi{10.2307/1971385}.

\bibitem[LL15]{LedrappierLim15}
F.~Ledrappier and S.~Lim.
\newblock Local limit theorem in negative curvature.
\newblock {\em ArXiv e-prints}, Mar. 2015,
  \burlalt{1503.04156}{http://arxiv.org/abs/1503.04156}.

\bibitem[LM84]{LyonsMcKean84}
T.~J. Lyons and H.~P. McKean.
\newblock Winding of the plane {B}rownian motion.
\newblock {\em Adv. in Math.}, 51(3):212--225, 1984.
\newblock \doi{10.1016/0001-8708(84)90007-0}.

\bibitem[Lot92]{Lott92}
J.~Lott.
\newblock Heat kernels on covering spaces and topological invariants.
\newblock {\em J. Differential Geom.}, 35(2):471--510, 1992.
\newblock \urlprefix\url{http://projecteuclid.org/euclid.jdg/1214448084}.

\bibitem[PS08]{PavliotisStuart08}
G.~A. Pavliotis and A.~M. Stuart.
\newblock {\em Multiscale methods -- Averaging and homogenization}, volume~53
  of {\em Texts in Applied Mathematics}.
\newblock Springer, New York, 2008.
\newblock \doi{10.1007/978-0-387-73829-1}.

\bibitem[PSC01]{PittetSaloffCoste01}
C.~Pittet and L.~Saloff-Coste.
\newblock A survey on the relationships between volume growth, isoperimetry,
  and the behavior of simple random walk on cayley graphs, with examples, 2001.
\newblock \urlprefix\url{http://www.math.cornell.edu/~lsc/surv.ps.gz}.

\bibitem[PY86]{PitmanYor86}
J.~Pitman and M.~Yor.
\newblock Asymptotic laws of planar {B}rownian motion.
\newblock {\em Ann. Probab.}, 14(3):733--779, 1986.
\newblock \urlprefix\url{http://www.jstor.org/stable/2244132}.

\bibitem[PY89]{PitmanYor89}
J.~Pitman and M.~Yor.
\newblock Further asymptotic laws of planar {B}rownian motion.
\newblock {\em Ann. Probab.}, 17(3):965--1011, 1989.
\newblock \urlprefix\url{http://www.jstor.org/stable/2244390}.

\bibitem[RH87]{RudnickHu87}
J.~Rudnick and Y.~Hu.
\newblock The winding angle distribution of an ordinary random walk.
\newblock {\em J. Phys. A}, 20(13):4421--4438, 1987.
\newblock \urlprefix\url{http://stacks.iop.org/0305-4470/20/4421}.

\bibitem[RS78]{ReedSimon78}
M.~Reed and B.~Simon.
\newblock {\em Methods of modern mathematical physics. {IV}. {A}nalysis of
  operators}.
\newblock Academic Press [Harcourt Brace Jovanovich, Publishers], New
  York-London, 1978.

\bibitem[RW00]{RogersWilliams00}
L.~C.~G. Rogers and D.~Williams.
\newblock {\em Diffusions, {M}arkov processes, and martingales. {V}ol. 1}.
\newblock Cambridge Mathematical Library. Cambridge University Press,
  Cambridge, 2000.
\newblock \doi{10.1017/CBO9781107590120}.
\newblock Foundations, Reprint of the second (1994) edition.

\bibitem[Spi58]{Spitzer58}
F.~Spitzer.
\newblock Some theorems concerning {$2$}-dimensional {B}rownian motion.
\newblock {\em Trans. Amer. Math. Soc.}, 87:187--197, 1958.
\newblock \doi{10.2307/1993096}.

\bibitem[Sun89]{Sunada89}
T.~Sunada.
\newblock Unitary representations of fundamental groups and the spectrum of
  twisted {L}aplacians.
\newblock {\em Topology}, 28(2):125--132, 1989.
\newblock \doi{10.1016/0040-9383(89)90015-3}.

\bibitem[SV71]{StroockVaradhan71}
D.~W. Stroock and S.~R.~S. Varadhan.
\newblock Diffusion processes with boundary conditions.
\newblock {\em Comm. Pure Appl. Math.}, 24:147--225, 1971.
\newblock \doi{10.1002/cpa.3160240206}.

\bibitem[TW95]{TobyWerner95}
E.~Toby and W.~Werner.
\newblock On windings of multidimensional reflected {B}rownian motion.
\newblock {\em Stochastics Stochastics Rep.}, 55(3-4):315--327, 1995.
\newblock \doi{10.1080/17442509508834030}.

\bibitem[Wat00]{Watanabe00}
S.~Watanabe.
\newblock Asymptotic windings of {B}rownian motion paths on {R}iemann surfaces.
\newblock {\em Acta Appl. Math.}, 63(1-3):441--464, 2000.
\newblock \doi{10.1023/A:1010756726463}.
\newblock Recent developments in infinite-dimensional analysis and quantum
  probability.

\bibitem[Wen17]{Wen17}
H.~Wen.
\newblock {\em Winding angle distribution of planar {M}arkov processes}.
\newblock PhD thesis, University of Wisconsin, Madison, Mathematics Dept.,
  Madison, WI, May 2017.

\end{thebibliography}
\end{document}